\documentclass[10pt]{article}
\usepackage[a4paper]{geometry}
\usepackage{lmodern}

\usepackage{authblk}
\usepackage[breaklinks]{hyperref}

\usepackage{amsmath,amsthm,amssymb,mathrsfs}

\usepackage{booktabs}
\usepackage{caption}		% So that one can use empty captions
\usepackage{footnote}		% For using footnotes inside tables
\makesavenoteenv{tabular}	% For using footnotes inside tables
\makesavenoteenv{table}		% For using footnotes inside tables

\usepackage{graphicx}
\usepackage{wrapfig}
\usepackage{enumerate}
\usepackage{enumitem}
\usepackage[nottoc]{tocbibind}

\theoremstyle{plain}
\newtheorem{theorem}{Theorem}
\newtheorem{lemma}[theorem]{Lemma} 
\newtheorem{prop}[theorem]{Proposition}

\theoremstyle{definition}
\newtheorem{definition}[theorem]{Definition}

\newtheorem*{remark}{Remark}
\newtheorem{example}[theorem]{Example}

\newcommand{\tri}{\mathcal{T}}
\newcommand{\hbody}{\mathcal{H}}
\newcommand{\manifold}{\mathcal{M}}
\newcommand{\surface}{\mathcal{S}}
\newcommand{\tree}{H}
\newcommand{\alttree}{T}
\newcommand{\altsurface}{\mathcal{F}}
\newcommand{\altaltsurface}{\mathcal{R}}
\newcommand{\problem}{\mathcal{P}}
\newcommand{\petersen}{P}
\newcommand{\compbodyone}{\mathcal{N}}
\newcommand{\compbodytwo}{\mathcal{K}}
\newcommand{\fork}{F}
\newcommand{\forkcomp}{\boldsymbol{F}}
\newcommand{\inputset}{\mathcal{I}}
\newcommand{\instance}{I}

\newcommand{\tw}[1]{\operatorname{tw} (#1)}
\newcommand{\pw}[1]{\operatorname{pw} (#1)}
\newcommand{\cw}[1]{\operatorname{cw} (#1)}
\newcommand{\cng}[1]{\operatorname{cng} (#1)}
\newcommand{\chd}[1]{\operatorname{chd} (#1)}
\newcommand{\Par}[1]{p (#1)}
\newcommand{\CW}[1]{\mathscr{L} (#1)}
\newcommand{\graph}[1]{\mathscr{G} (#1)}

\newbox{\myorcidaffilbox}
\sbox{\myorcidaffilbox}{\large\includegraphics[height=0.9em]{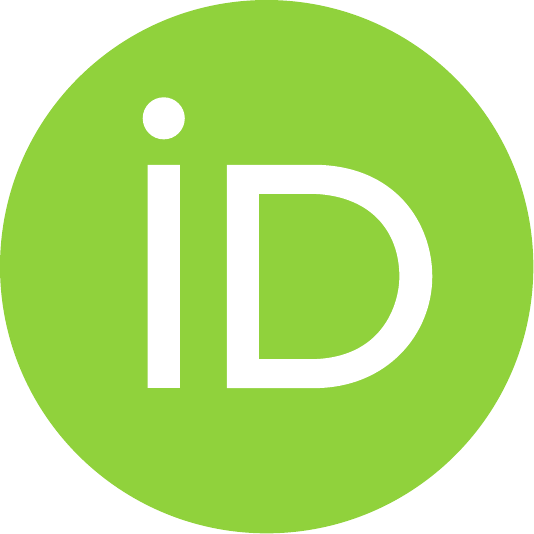}}
\newcommand{\orcidaffil}[1]{%
  \href{https://orcid.org/#1}{\usebox{\myorcidaffilbox}}}

\title{On the treewidth of triangulated $3$-manifolds\thanks{Apart from the formatting and updated references, this manuscript is identical to the final version published in the Journal of Computational Geometry \cite{huszar2019treewidth}. Author affiliations correspond to the time of publication.}}
\author[$\dagger$]{Krist\'of Husz\'ar\,\orcidaffil{0000-0002-5445-5057}\,}
\author[$\ddagger$]{Jonathan Spreer\,\orcidaffil{0000-0001-6865-9483}\,}
\author[$\dagger$]{Uli Wagner\,\orcidaffil{0000-0002-1494-0568}\,}

\affil[$\dagger$]{\normalsize Institute of Science and Technology Austria, Klosterneuburg, Austria}
\affil[$\ddagger$]{\normalsize Institut f\"ur Mathematik, Freie Universit\"at Berlin, Germany}

\date{}

\begin{document}
\maketitle

\begin{abstract}
In graph theory, as well as in $3$-manifold topology, there exist several width-type parameters to describe how ``simple'' or ``thin'' a given graph or $3$-manifold is. These parameters, such as pathwidth or treewidth for graphs, or the concept of thin position for $3$-manifolds, play an important role when studying algorithmic problems; in particular, there is a variety of problems in computational $3$-manifold topology---some of them known to be computationally hard in general---that become solvable in polynomial time as soon as the dual graph of the input triangulation has bounded treewidth.

In view of these algorithmic results, it is natural to ask whether every $3$-manifold admits a triangulation of bounded treewidth. We show that this is not the case, i.e., that there exists an infinite family of closed $3$-manifolds not admitting triangulations of bounded pathwidth or treewidth (the latter implies the former, but we present two separate proofs).

We derive these results from work of Agol, of Scharlemann and Thompson, and of Scharlemann, Schultens and Saito by exhibiting explicit connections between the topology of a $3$-manifold $\manifold$ on the one hand and width-type parameters of the dual graphs of triangulations of  $\manifold$ on the other hand, answering a question that had been raised repeatedly by researchers in computational $3$-manifold topology.
In particular, we show that if a closed, orientable, irreducible, non-Haken $3$-manifold $\manifold$ has a triangulation of treewidth (resp.\ pathwidth) $k$ then the Heegaard genus of $\manifold$ is at most $18(k+1)$ (resp.\ $4(3k+1)$).
\end{abstract}

\section{Introduction}
\label{sec:intro}

In the field of $3$-manifold topology many fundamental problems can be solved algorithmically. Famous examples include deciding whether a given knot is trivial \cite{haken1961normal}, deciding whether a given $3$-manifold is homeomorphic to the $3$-sphere \cite{rubinstein1995algorithm,thompson1994thin}, and, more generally (based on Perelman's proof of Thurston's geometrization conjecture \cite{kleiner2008perelman}), deciding whether two given $3$-manifolds are homeomorphic, see, e.g., \cite{bessieres2010geometrisation,kuperberg2019algorithmic,scott2014homeomorphism}. The algorithm for solving the homeomorphism problem is still purely theoretical, and its complexity remains largely unknown \cite{kuperberg2019algorithmic,lackenby2017conditionally}. In contrast, the first two problems are known to lie in the intersection of the complexity classes {\bf NP} and {\bf co-NP} \cite{hass1999computational,ivanov2008computational,kuperberg2014knottedness,lackenby2016efficient,schleimer2011sphere,zentner2018integer}.\footnote{The proof of {\bf co-NP} membership for 3-sphere recognition assumes the Generalized Riemann Hypothesis.}

Moreover, implementations of, for instance, algorithms to recognize the $3$-sphere exist out-of-the-box (e.g., using the computational $3$-manifold software {\em Regina} \cite{regina}) and exhibit practical running times for virtually all known inputs.

In fact, many topological problems with implemented algorithmic solutions solve problem instances of considerable size. This is despite the fact that  most of these implementations have prohibitive worst-case running times, or the underlying problems are even known to be computationally hard in general. In recent years, there have been several attempts to explain this gap using the concepts of parameterized complexity and algorithms for fixed parameter tractable (FPT) problems \cite{downey1999parameterized,downey2013fundamentals}. This effort has proven to be highly effective and, today, there exist numerous FPT algorithms in the field \cite{burton2017courcelle, burton2016parameterized, burton2018algorithms, burton2013complexity, maria2019polynomial}. More specifically, given a triangulation $\tri$ of a $3$-manifold $\manifold$ with $n$ tetrahedra whose dual graph $\Gamma (\tri)$ has treewidth\footnote{We often simply speak of the treewidth of a triangulation, meaning the treewidth of its dual graph.} at most $k$, there exist algorithms to compute

\begin{itemize}
	\item taut angle structures\footnote{Taut angle structures are combinatorial versions of semi-simplicial metrics which have implications on the geometric properties of the underlying manifold.} of what is called ideal triangulations with torus boundary components in running time $O(7^k \cdot n)$ \cite{burton2013complexity};
	\item optimal Morse matchings\footnote{Optimal Morse matchings translate to discrete Morse functions with the minimum number of critical points with respect to the combinatorics of the triangulation and the topology of the underlying $3$-manifold.} in the Hasse diagram of $\tri$ in $O(4^{k^2+k} \cdot k^3 \cdot \log{k} \cdot n)$ \cite{burton2016parameterized};
	\item the Turaev--Viro invariants\footnote{Turaev--Viro invariants are powerful tools to distinguish between $3$-manifolds. They are the method of choice when, for instance, creating large censuses of manifolds.} for parameter $r \geq 3$ in $O((r-1)^{6(k+1)} \cdot k^2 \cdot \log{r} \cdot n)$ \cite{burton2018algorithms};
	\item every problem which can be expressed in monadic second-order logic in $O(f(k) \cdot n)$, where $f$ often is a tower of exponentials \cite{burton2017courcelle}.\footnote{This result is analogous to Courcelle's celebrated theorem in graph theory \cite{courcelle1990monadic}.}
\end{itemize}

Some of these results are not purely theoretical (as is sometimes the case with FPT algorithms) but are implemented and outperform previous state-of-the-art implementations for typical input. As a result, they have a significant practical impact. This is in particular the case for the algorithm to compute Turaev--Viro invariants \cite{burton2018algorithms,maria2019polynomial}. 

Note that treewidth (the dominating factor in the running times given above) is a combinatorial quantity linked to a triangulation, not a topological invariant of the underlying manifold. This gives rise to the following approach to efficiently solve topological problems on a $3$-manifold $\manifold$: given a triangulation $\tri$ of $\manifold$, search for a triangulation $\tri'$ of the same manifold with smaller treewidth. 

This approach faces severe difficulties. By a theorem due to Kirby and Melvin \cite{kirby2004local}, the Turaev--Viro invariant for parameter $r=4$ is ${\bf \# P}$-hard to compute. Thus, if there were a polynomial time procedure to turn an $n$-tetrahedron triangulation $\tri$ into a $\operatorname{poly} (n)$-tetrahedron triangulation $\tri'$ with dual graph of treewidth at most $k$, for some universal constant $k$, then this procedure, combined with the algorithm from \cite{burton2018algorithms}, would constitute a polynomial time solution for a ${\bf \# P}$-hard problem. Furthermore, known facts imply that \emph{most} triangulations of \emph{most} $3$-manifolds must have large treewidth\footnote{It is known that, given $k \in \mathbb{N}$, there exist constants $C,C_k > 1$ such that there are at least $C^{n\log(n)}$ 3-manifolds which can be triangulated with $\leq n$ tetrahedra, whereas there are at most $C_k^n$ triangulations with treewidth $\leq k$ and $\leq n$ tetrahedra.} (see Propositions~\ref{prop:manyMfds} and \ref{prop:fewSmallTW} in Appendix~\ref{app:most}).
However, while these arguments indicate that triangulations of small treewidth may be rare and computationally hard to find, they do not rule out that every manifold has some (potentially very large) triangulation of bounded treewidth.

In this article we show that this is actually not the case, answering a question that had been raised repeatedly by researchers in computational $3$-manifold topology.\footnote{The question whether every 3-manifold admits a triangulation of bounded treewidth, and variations thereof have been asked at several meetings and open problem sessions including an Oberwolfach meeting in 2015 \cite[Problem 8]{MFOReports} (formulated in the context of knot theory).} More specifically, we prove the following two statements.

\begin{theorem}
\label{thm:pw}
There exists an infinite family of $3$-manifolds which does not admit triangulations with dual graphs of uniformly bounded pathwidth.
\end{theorem}

\begin{theorem}
\label{thm:tw}
There exists an infinite family of $3$-manifolds which does not admit triangulations with dual graphs of uniformly bounded treewidth.
\end{theorem}

We establish the above results through the following theorems, which are the main contributions of the present paper. The necessary terminology is introduced in Section \ref{sec:prelims}.

\begin{theorem}
\label{thm:heeg_pw}
Let $\manifold$ be a closed, orientable, irreducible, non-Haken $3$-manifold and let $\tri$ be a triangulation of $\manifold$ with dual graph $\Gamma(\tri)$ of pathwidth $\pw{\Gamma(\tri)} \leq k$. Then $\manifold$ has Heegaard genus $\mathfrak{g}(\manifold) \leq 4(3k+1)$.
\end{theorem}

\begin{theorem}
\label{thm:heeg_tw}
Let $\manifold$ be a closed, orientable, irreducible, non-Haken $3$-manifold and let $\tri$ be a triangulation of $\manifold$ with dual graph $\Gamma(\tri)$ of treewidth $\tw{\Gamma (\tri)} \leq k$. Then $\manifold$ has Heegaard genus $\mathfrak{g}(\manifold) \leq 18(k+1)$.
\end{theorem}

By a result of Agol \cite{agol2003small} (Theorem \ref{thm:a} in this paper), there exist closed, orientable, irreducible, non-Haken $3$-manifolds of arbitrarily large Heegaard genus. Combining this result with Theorems~\ref{thm:heeg_pw} and \ref{thm:heeg_tw} thus immediately implies Theorems \ref{thm:pw} and \ref{thm:tw}.

\begin{remark}
Note that Theorem~\ref{thm:pw} can be directly deduced from Theorem~\ref{thm:tw} since the pathwidth of a graph is always at least as large as its treewidth.\footnote{This is immediate from the definitions of treewidth and pathwidth, see Section~\ref{sec:params}.} Nonetheless, we provide separate proofs for each of the two statements. The motivation is that while the proof of Theorem~\ref{thm:heeg_pw} is considerably simpler than that of Theorem~\ref{thm:heeg_tw}, it already illustrates several key concepts and ideas which we are building upon in the proof of Theorem~\ref{thm:heeg_tw}.
\end{remark}

The paper is organized as follows. After going over the preliminaries in Section~\ref{sec:prelims}, we give an overview of selected width-type graph parameters in Section~\ref{sec:params}. Most notably, we propose the {\em congestion} of a graph (also known as {\em carving width}) as an alternative choice of a parameter for FPT algorithms in $3$-manifold topology (cf.\ Appendix~\ref{app:fpt}). Section~\ref{sec:3mfdParams} is devoted to results from 3-manifold topology which we build upon. In Section~\ref{sec:lwpw} we then prove Theorem~\ref{thm:pw}, and in Section~\ref{sec:toptw} we prove Theorem~\ref{thm:tw}.

\paragraph*{Acknowledgments.} We are grateful to Arnaud de Mesmay and Saul Schleimer for noticing a gap in an earlier proof of Theorem \ref{thm:heeg_tw} and for pointing us towards key ingredients of the current proof.
Moreover, we would like to thank Jessica Purcell, Jennifer Schultens and Stephan Tillmann for helpful discussions, and the anonymous reviewers for useful comments and suggestions regarding the exposition.

This work was initiated during a visit of the second author at IST Austria. The second author would like to thank the people at IST Austria for their hospitality.
Research of the second author was supported by the Einstein Foundation (project ``Einstein Visiting Fellow Santos'') and by the Simons Foundation (``Simons Visiting Professors'' program).

\newpage

\section{Preliminaries}
\label{sec:prelims}

In this section we recall some basic concepts and terminology of graph theory, $3$-manifolds, triangulations, and parameterized complexity theory. 

\paragraph*{Graphs vs.\ triangulations.} Following several authors in the field, we use the terms {\em edge} and {\em vertex} to refer to an edge or vertex in a $3$-manifold triangulation, whereas the terms {\em arc} and {\em node} denote an edge or vertex in a graph, respectively.

\subsection{Graphs}

For general background on graph theory we refer to \cite{diestel2017graph}.

A {\em graph} (more specifically, a {\em multigraph}) $G=(V,E)$ is an ordered pair consisting of a finite set $V=V(G)$ of {\em nodes} and a multiset $E=E(G)$ of unordered pairs of nodes, called {\em arcs}. We allow {\em loops}, i.e., an arc $e \in E$ might itself be a multiset, e.g., $e = \{v,v\}$ for some $v \in V$.
The {\em degree} of a node $v \in V$, denoted by $\deg(v)$, equals the number of arcs containing it, counted with multiplicity. In particular, a loop $\{v,v\}$ contributes two to the degree of $v$. 
For every node $v \in V$ of a graph $G$, its {\em star} $\operatorname{st}_G (v)$ denotes the set of edges incident to $v$.
A graph is called {\em $k$-regular} if all of its nodes have the same degree $k \in \mathbb{N}$. A {\em tree} is a connected graph with $n$ nodes and $n-1$ arcs.
A node of degree one is called a {\em leaf}.

\subsection{3-Manifolds and their triangulations}
\label{ssec:3mfds}

For an introduction to the topology and geometry of $3$-manifolds and to their triangulations we refer to the textbook \cite{schultens2014introduction} and to the seminal monograph \cite{thurston2014three}.

A {\em $3$-manifold with boundary} is a topological space\footnote{More precisely, we only consider topological spaces which are second countable and Hausdorff.} $\manifold$ such that each point $x \in \manifold$ has a neighborhood which either looks like (i.e., is homeomorphic to) the Euclidean $3$-space $\mathbb{R}^3$ or the closed upper half-space $\{(x,y,z) \in \mathbb{R}^3 : z\geq 0\}$.
The points of $\manifold$ that do not have a neighborhood homeomorphic to $\mathbb{R}^3$ constitute the {\em boundary $\partial \manifold$} of $\manifold$. A compact $3$-manifold is {\em closed} if it has an empty boundary.

Informally, two $3$-manifolds are equivalent (or {\em homeomorphic}) if one can be turned into the other by a continuous, reversible deformation.
In other words, when talking about a $3$-manifold, we are not interested in its particular shape, but only in its qualitative properties (topological invariants), such as ``number of boundary components'', or ``connectedness''.

All $3$-manifolds considered in this article are assumed to be compact and orientable.

\paragraph*{Handle decompositions.} Every compact $3$-manifold can be built from finitely many building blocks called $0$-, $1$-, $2$-, and $3$-handles. In such a {\em handle decomposition} all handles are (homeomorphic to) $3$-balls, and are only distinguished in how they are glued to the existing decomposition.
For instance, to build a closed $3$-manifold from handles, we may start with a disjoint union of $3$-balls, or {\em $0$-handles}, where further $3$-balls are glued to the boundary of the existing decomposition along pairs of $2$-dimensional disks, the so-called {\em $1$-handles}, or along annuli, the so-called {\em $2$-handles}. This process is iterated until the boundary of the decomposition consists of a union of $2$-spheres. These are then eliminated by gluing in one additional $3$-ball per boundary component, the {\em $3$-handles} of the decomposition. 

In every step of building up a (closed) $3$-manifold $\manifold$ from handles, the existing decomposition is a submanifold whose boundary---called a {\em bounding surface}---separates $\manifold$ into two pieces: the part that is already present, and its complement (each of them possibly disconnected). 

Bounding surfaces and, more generally, all kinds of surfaces embedded in a $3$-manifold, play an important role in the study of $3$-manifolds (similar to that of simple closed curves in the study of surfaces). When chosen carefully, an embedded surface reveals valuable information about the topology of the ambient $3$-manifold.

\paragraph{Surfaces in 3-manifolds.} Given a $3$-manifold $\manifold$, a surface $\surface\subset \manifold$ is said to be {\em properly embedded}, if it is embedded in $\manifold$, and for the boundary we have $\partial \surface = \surface \cap \partial \manifold$. Let $\surface \subset \manifold$ be a properly embedded surface distinct from the $2$-sphere, and let $D$ be a disk embedded into $\manifold$ such that its boundary satisfies $\partial D = D \cap \surface$. $D$ is said to be a {\em compressing disk for $\surface$} if $\partial D$ does not bound a disk on $\surface$.  If such a compressing disk exists, then $\surface$ is called {\em compressible}, otherwise it is called {\em incompressible}. 
An embedded $2$-sphere $\surface \subset \manifold$ is called {\em incompressible} if $\surface$ does not bound a $3$-ball in~$\manifold$.\footnote{A standard example of a compressible surface is a torus (or any other orientable surface) embedded in the $3$-sphere $\mathbb{S}^3$, and of an incompressible surface is the $2$-sphere $ \mathbb{S}^2 \times \{ x \} \subset \mathbb{S}^2 \times \mathbb{S}^1$.}

A 3-manifold $\manifold$ is called {\em irreducible}, if every embedded $2$-sphere bounds a $3$-ball in $\manifold$.
Moreover, it is called {\em $P^2$-irreducible}, if it does not contain an embedded $2$-sided\footnote{A properly embedded surface $\surface \subset \manifold$ is $2$-sided in $\manifold$, if the codimension zero submanifold in $\manifold$ obtained by thickening $\surface$ has two boundary components, i.e., $\surface$ locally separates $\manifold$ into two pieces.} real projective plane $\mathbb{R}P^2$. This notion is only significant for non-orientable manifolds, since orientable $3$-manifolds cannot contain any $2$-sided non-orientable surfaces, and are readily $P^2$-irreducible. If a $P^2$-irreducible, irreducible $3$-manifold $\manifold$ contains a $2$-sided incompressible surface, then it is called {\em Haken}, otherwise it is called {\em non-Haken}.

Finally, let $\surface$ be a closed and orientable surface (possibly disconnected). The {\em genus} of $\surface$, denoted by $g(\surface)$, equals the maximum number of pairwise disjoint simple closed curves one can remove from $\surface$ without increasing the number of connected components.

\paragraph*{Handlebodies and compression bodies.} We have already discussed handle decompositions of $3$-manifolds. Closely related are the notions of handlebody and compression body.

A {\em handlebody} $\hbody$ is a connected $3$-manifold with boundary which can be described as a single $0$-handle with a number of $1$-handles attached to it, or, equivalently, as a thickened graph. Up to homeomorphism, $\hbody$ is determined by the genus $g(\partial\hbody)$ of its boundary.

Let $\surface$ be a closed, orientable (not necessarily connected) surface. A {\em compression body} is a $3$-manifold $\compbodyone$ obtained from $\surface \times [0,1]$ by attaching $1$-handles to $\surface \times \{1\}$, and filling in some of the $2$-sphere components of $\surface \times \{0\}$ with $3$-balls. $\compbodyone$ has two sets of boundary components: $\partial_- \compbodyone = \surface \times \{0\} \setminus \{\text{filled-in 2-sphere components}\}$ and $\partial_+ \compbodyone = \partial \compbodyone \setminus \partial_- \compbodyone$.

Dual to this construction, a compression body can be built by starting with a closed, orientable surface $\altsurface$, thickening it to $\altsurface \times [0,1]$, attaching $2$-handles along $\altsurface \times \{0\}$, and filling in some of the resulting $2$-spheres with $3$-balls.
This 3-manifold $\compbodyone$ has again two sets of boundary components given by $\partial_{+} \compbodyone = \altsurface \times \{1\}$ and $\partial_{-} \compbodyone = \partial \compbodyone \setminus \partial_{+} \compbodyone$. 

In accordance with \cite{hoffoss2016morse, hoffoss2017morse}, we call $\partial_{+} \compbodyone$ the {\em top boundary}, and $\partial_{-} \compbodyone$ the {\em lower boundary} of $\compbodyone$. Note that, by construction, we always have $g(\partial_{+} \compbodyone) \geq g(\partial_{-} \compbodyone)$. Moreover, if $\partial_- \compbodyone = \emptyset$, then the compression body $\compbodyone$ is actually a handlebody.

\paragraph*{Heegaard splittings.} Let $\manifold$ be a $3$-manifold, possibly with boundary. A {\em Heegaard splitting} of $\manifold$ is a decomposition $\manifold = \compbodyone \cup_{\surface} \compbodytwo$ (i.e., $\compbodyone \cup\compbodytwo = \manifold$ and $\compbodyone \cap \compbodytwo = \surface$) into compression bodies $\compbodyone$ and $\compbodytwo$ with $\surface = \partial_+\compbodyone = \partial_+\compbodytwo$ and $\partial\manifold = \partial_-\compbodyone \cup \partial_-\compbodytwo$.
The {\em Heegaard genus} of $\manifold$, denoted $\mathfrak{g}(\manifold)$, is the minimum possible genus $g(\surface)$ over all Heegaard splittings of $\manifold$.

A fundamental result of Moise \cite{moise1952affine} implies that every compact orientable 3-manifold admits a Heegaard splitting (also see the survey \cite{scharlemann2002heegaard}).

\begin{example}[Heegaard splittings from handle decompositions]
\label{ex:heegaard_splitting}
When building up a closed connected $3$-manifold $\manifold$ from handles, one may assume that (possibly after deforming the attaching maps) all $0$- and $1$-handles are attached before any $2$- or $3$-handles.
Defining $\hbody_1$ to be the union of the $0$- and $1$-handles, and $\hbody_2$ to be the union of the $2$- and $3$-handles yields a Heegaard splitting $\manifold = \hbody_1 \cup_{\surface} \hbody_2$, $\surface=\partial\hbody_1=\partial\hbody_2$, into handlebodies $\hbody_1$ and $\hbody_2$.
\end{example}

\begin{example}[Lens spaces]
\label{ex:lens}
Among the best known $3$-manifolds are the closed orientable 3-manifolds of Heegaard genus one. These manifolds are also known under the name of {\em lens spaces}. 
To construct them, let $p,q$ be two positive co-prime integers. The lens space $\operatorname{L}(p,q)$ is then obtained by taking two solid tori $\mathbb{T}_i=\mathbb{S}^1 \times D$, $i=1,2$, and gluing them together along their boundaries in a way such that the meridian $m_1 = \{x_1\}\times \partial D \subset \mathbb{T}_1$ of $\mathbb{T}_1$ is mapped onto the curve of $\mathbb{T}_2$ which wraps $p$ times around the longitude $l_2 = \mathbb{S}^1 \times \{y_2\} \subset \mathbb{T}_2$ and  $q$ times around the meridian $m_2 = \{x_2\}\times \partial D \subset \mathbb{T}_2$.
\end{example}

\paragraph*{Linear splittings.} In their work on thin position \cite{scharlemann1992thin}, discussed in Section~\ref{sec:3mfdParams}, Scharlemann and Thompson consider a generalization of Heegaard splittings, we call a {\em linear splitting}, which arises naturally from more complicated sequences of handle attachments, e.g., when building up a manifold by first only attaching some of the $0$- and $1$-handles before attaching $2$- and $3$-handles (cf.\ Example \ref{ex:linear_splitting}).\footnote{While this construction is sometimes called a {\em generalized Heegaard splitting}, we prefer the more expressive term of a {\em linear splitting} \cite{hoffoss2016morse} to make a distinction from the even more general {\em graph splittings} \cite{hoffoss2017morse}.}

\noindent A {\em linear splitting} of a $3$-manifold $\manifold$ is a decomposition
\[
	\manifold = \left(\compbodyone_1 \cup_{\altaltsurface_1} \compbodytwo_1\right) \cup_{\altsurface_1} \left(\compbodyone_2 \cup_{\altaltsurface_2} \compbodytwo_2\right) \cup_{\altsurface_2} \cdots \cup_{\altsurface_{s-1}} \left(\compbodyone_s \cup_{\altaltsurface_s} \compbodytwo_s\right),
\]
where $(\compbodyone_1, \compbodytwo_1, \ldots , \compbodyone_s, \compbodytwo_s)$ is a sequence of (possibly disconnected) compression bodies in $\manifold$. They are pairwise disjoint except for subsequent pairs, which are ``glued together'' along (pairwise disjoint) closed surfaces $\altaltsurface_1,\altsurface_1,\ldots,\altaltsurface_{s-1},\altsurface_{s-1},\altaltsurface_s$ in $\manifold$. More precisely,
\[
	\altaltsurface_i = \compbodyone_i \cap \compbodytwo_i = \partial_+\compbodyone_i = \partial_+\compbodytwo_i \quad \text{and} \quad \altsurface_j = \compbodytwo_j \cap \compbodyone_{j+1} = \partial_-\compbodytwo_j = \partial_-\compbodyone_{j+1}.
\]
The lower boundaries of $\compbodyone_1$ and $\compbodytwo_s$ constitute the boundary of $\manifold$, i.e., $\partial\manifold = \partial_-\compbodyone_1 \cup \partial_-\compbodytwo_s$.

\begin{example}[Linear splittings from handle decompositions]
\label{ex:linear_splitting}
Assume $\manifold$ is a closed $3$-manifold given via a handle decomposition, i.e., a sequence of handle attachments to build up $\manifold$. 
Consider the first terms of the sequence up to (but not including) the first $2$- or $3$-handle attachment. Let $\compbodyone_1$ be the union of all handles in this subsequence.
In the second step look at all $2$- and $3$-handles following the initial sequence of $0$- and $1$-handles until we reach $0$- or $1$-handles again, and follow the dual construction to obtain another compression body $\compbodytwo_1$. More precisely, define $\partial_{+} \compbodytwo_1 = \partial_{+} \compbodyone_1$, thicken the top boundary into $\partial_{+} \compbodytwo_1 \times [0,1]$, and then attach the given $2$- and $3$-handles along $\partial_{+} \compbodytwo_1 \times \{0\}$.
Iterating this procedure results in a linear splitting of $\manifold$ into compression bodies $(\compbodyone_1, \compbodytwo_1, \ldots , \compbodyone_s, \compbodytwo_s)$.
\end{example}

\paragraph*{Graph splittings and fork complexes.} The decomposition described above exhibits a linear structure.
Here we introduce a more general approach of decomposing a $3$-manifold into compression bodies following a graph structure~\cite{scharlemann2016lecture}. 

The main difference is that now we allow the lower boundary components of a compression body to be glued to lower boundary components of distinct compression bodies.
The top boundary of a compression body, however, is still identified with the top boundary of a single other compression body.
This structure can be represented by a so-called {\em fork complex} (which is essentially a labeled graph) in which the compression bodies of the decomposition are modeled by {\em fork}s.
More precisely, an {\em $n$-fork} is a tree $\fork$ with $n+2$ nodes $V(\fork)=\{g,p,t_1,\ldots,t_n\}$ with $p$ being of degree $n+1$ and all other nodes being leaves.
The nodes $g$, $p$, and the $t_i$ are called the $grip$, the $root$, and the $tine$s of $\fork$, respectively (Figure \ref{fig:forkcomp_examples}(i) shows a $0$- and a $3$-fork).
We think of a fork $\fork = \fork_\compbodyone$ as an abstraction of a compression body $\compbodyone$, such that the grip of $\fork$ corresponds to $\partial_+\compbodyone$, whereas the tines correspond to the connected components of $\partial_-\compbodyone$.

Informally, a {\em fork complex} $\forkcomp$ (representing a given decomposition of a $3$-manifold $\manifold$) is obtained by taking several forks (corresponding to the compression bodies which constitute $\manifold$), and identifying grips with grips, and tines with tines (following the way the boundaries of these compression bodies are glued together). The set of grips and tines which remain unpaired is denoted by $\partial\forkcomp$ (as they correspond to surfaces which constitute the boundary $\partial\manifold$).
The formal relationship between $\forkcomp$ and the underlying $3$-manifold $\manifold$ (possibly with boundary) is described by a map $\rho : (\manifold; \partial \manifold) \to (\forkcomp; \partial\forkcomp)$ which has to satisfy certain natural criteria \cite[Definition 5.1.7]{scharlemann2016lecture}. The pair $(\forkcomp,\rho)$ is called a {\em graph splitting}.
See Figure \ref{fig:forkcomp_examples} for illustrations, and \cite[Section 5.1]{scharlemann2016lecture} for further details.

\begin{figure}
	\centering
	\includegraphics[width=1\textwidth]{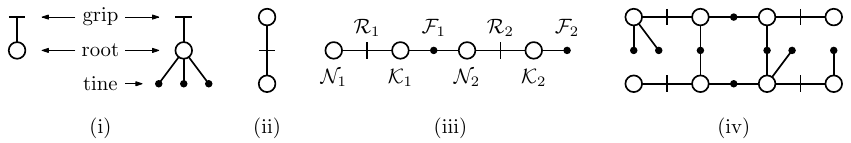}
	\caption{Fork complexes describing Heegaard (ii), linear (iii), and graph splittings (iv)}
	\label{fig:forkcomp_examples}
\end{figure}

\paragraph*{Triangulations.} In this article we typically describe $3$-manifolds by {\em triangulations} (also referred to as {\em generalised}, {\em semi-simplicial}, or {\em singular} triangulations in the literature).\footnote{Triangulations, in the present sense, provide a very efficient way to describe $3$-manifolds: more than $11,000$ topologically distinct $3$-manifolds can be triangulated with $11$ tetrahedra or less \cite{burton2011detecting,matveev2007algorithmic}. This efficiency comes at a price: we allow self-identifications (e.g., gluings of pairs of triangular faces of the same tetrahedron), and thus a triangulation is generally non-regular when seen as a (simplicial) cell-complex. However, this deficiency can be overcome by passing to the first barycentric subdivision. A second barycentric subdivision then yields a simplicial complex. In particular, every triangulation can be turned into a simplicial complex of size at most $24^2 = 576$ times larger than the original triangulation.

The aforementioned theorem of Moise \cite{moise1952affine} tells us that every $3$-manifold has a triangulation.}
That is, a finite collection of abstract tetrahedra, glued together in pairs along their triangular faces (called {\em triangles}). As a result of these {\em face gluings}, many tetrahedral edges (or vertices) are glued together and we refer to the result as a single {\em edge (or vertex) of the triangulation}.
A triangulation $\tri$ describes a closed $3$-manifold if no tetrahedral edge is identified with itself in reverse, and the boundary of a small neighborhood around each vertex is a $2$-sphere. Similarly, $\tri$ describes a $3$-manifold with boundary if, in addition, the boundary of small neighborhoods around some of the vertices are disks.

Given a triangulation $\tri$ of a closed $3$-manifold, its {\em dual graph $\Gamma (\tri)$} (also called the {\em face pairing graph}) is the graph with one node per tetrahedron of $\tri$, and with an arc between two nodes for each face gluing between the corresponding pair of tetrahedra. By construction, the dual graph is a 4-regular multigraph. 
Since every triangulation $\tri$ can be linked to its dual graph $\Gamma (\tri)$ this way, we often attribute properties of $\Gamma (\tri)$ directly to $\tri$.

\subsection{Parameterized complexity and fixed parameter tractability}

There exist various notions and concepts of a refined complexity analysis for theoretically difficult problems. Parameterized complexity, due to Downey and Fellows \cite{downey1999parameterized,downey2013fundamentals}, identifies a parameter on the set of inputs, which is responsible for the hardness of a given problem.

More precisely, for a problem $\problem$ with input set $\inputset$, a parameter is a (computable) function $p\colon \inputset \to \mathbb{N}$. If the parameter $p$ is the output of $\problem$, then $p$ is called the {\em natural parameter}. The problem $\problem$ is said to be {\em fixed parameter tractable for parameter $p$} (or {\em FPT in $p$} for short) if there exists an algorithm which solves $\problem$ for every instance $\instance \in \inputset$ with running time $O(f(p(\instance)) \cdot \operatorname{poly}(n))$, where $n$ is the size of the input $\instance$, and $f\colon \mathbb{N} \to \mathbb{N}$ is arbitrary (computable). By definition, such an algorithm then runs in polynomial time on the set of inputs with bounded $p$. Hence, this identifies, in some sense, $p$ as a potential ``source of hardness'' for~$\problem$, cf.\ the results listed in the \hyperref[sec:intro]{Introduction}.

In computational $3$-manifold topology, a very important set of parameters is the one of topological invariants, i.e., properties which only depend on the topology of a given manifold and are independent of the choice of triangulation (see \cite{maria2019polynomial} for such a result, using the first Betti number as parameter). However, most FPT-results in the field use parameters of the dual graph of a triangulation which greatly depend on the choice of the triangulation: every $3$-manifold admits a triangulation with arbitrarily high graph parameters---for all parameters considered in this article. 
The aim of this work is to link these parameters to topological invariants in the only remaining sense: given a $3$-manifold $\manifold$, find lower bounds for graph parameters of dual graphs of triangulations ranging over all triangulations of~$\manifold$.

\section{Width-type graph parameters}
\label{sec:params}

The theory of parameterized complexity has its sources in graph theory, where many problems which are ${\bf NP}$-hard in general become tractable in polynomial time if one assumes structural restrictions about the possible input graphs.
For instance, several graph theoretical questions have a simple answer if one asks them about trees, or graphs which are similar to trees in some sense.
Width-type parameters make this sense of similarity precise \cite{hlinveny2008width}.
We are particularly interested in the behavior of these parameters and their relationship with each other when considering bounded-degree graphs or, more specifically, dual graphs of $3$-manifold triangulations. (See Appendix~\ref{app:fpt} for computational aspects.)

\paragraph*{Treewidth and pathwidth.} The concepts of treewidth and pathwidth were introduced by Robertson and Seymour in their early papers on graph minors \cite{robertson1983graph, robertson1986graph}, also see the surveys \cite{bodlaender1994tourist,bodlaender2005discovering, bodlaender2008combinatorial}.
Given a graph $G$, its \emph{treewidth} $\tw{G}$ measures how tree-like the graph is.

\begin{definition}[Tree decomposition, treewidth] A \textit{tree decomposition} of a graph $G=(V,E)$ is a tree $\alttree$ with nodes $B_1,\ldots,B_m \subseteq V$, also called {\em bags}, such that
\begin{enumerate}
\itemsep0em 
	\item \label{twpropone} $B_1 \cup \ldots \cup B_m = V$,
	\item \label{twproptwo} if $v \in B_i \cap B_j$ then $v \in B_k$ for all bags $B_k$ of $\alttree$ in the path between $B_i$ and $B_j$, in other words, the bags containing $v$ span a (connected) subtree of $\alttree$,
	\item \label{twpropthree} for every arc $\{u,v\} \in E$, there exists a node $B_i$ such that $\{u,v\} \subseteq B_i$.
\end{enumerate}

The \textit{width} of a tree decomposition equals the size of the largest bag minus one. The \emph{treewidth} $\tw{G}$ is the minimum width among all possible tree decompositions of $G$.
\label{defn-treewidth}
\end{definition}

\begin{definition}[Path decomposition, pathwidth] A \emph{path decomposition} of  a graph $G=(V,E)$ is a tree decomposition for which the tree $\alttree$ is required to be a \emph{path}. The  \emph{pathwidth} $\pw{G}$ of a graph $G$ is the minimum width of any path decomposition of~$G$.
\end{definition}

\paragraph*{Cutwidth.} The \emph{cutwidth} $\cw{G}$ of a graph $G$ is the graph-analogue of the linear width of a manifold (to be discussed in Section~\ref{sec:3mfdParams}). If we order the nodes $\{v_1, \ldots , v_n \} = V(G)$ of $G$ on a line, the set of arcs running from a node $v_i$, $i \leq \ell$, to a node $v_j$, $j > \ell$, is called a cutset $C_{\ell}$ of the ordering. The cutwidth $\cw{G}$ is defined to be the cardinality of the largest cutset, minimized over all linear orderings of $V(G)$.

Cutwidth and pathwidth are closely related: for bounded-degree graphs they are within a constant factor. Let $\Delta(G)$ denote the maximum degree of a node in~$G$.

\begin{theorem}[Bodlaender, Theorems 47 and 49 from \cite{bodlaender1998partial}\footnote{The inequality $\cw{G} \leq \Delta(G)\pw{G}$ seems to be already present in the earlier work of Chung and Seymour \cite{chung1989graphs} on the relation of cutwidth to another parameter called topological bandwith (see Theorem 2 in \cite{chung1989graphs}). Pathwidth plays an intermediate, connecting role there. However, the inequality is phrased and proved explicitly by Bodlaender in \cite{bodlaender1998partial}.}]
\label{thm:parameters}
Given a graph $G$, we have 
\[
	\pw{G} \leq \cw{G} \leq \Delta(G)\pw{G}.
\]
\end{theorem}

\paragraph*{Congestion.} In \cite{bienstock1989graph} Bienstock introduced {\em congestion}, a generalization of cutwidth, which is a quantity related to treewidth in a similar way as cutwidth to pathwidth (compare Theorems \ref{thm:parameters} and \ref{thm:tw-cng-tw}).

Let us consider two graphs $G$ and $\tree$, called the {\em guest} and the {\em host}, respectively. An {\em embedding} $\mathcal{E} = (\iota,\rho)$ of $G$ into $\tree$ consists of an injective mapping $\iota\colon V(G)\rightarrow V(\tree)$ together with a routing $\rho$ that assigns to each arc $\{u,v\} \in E(G)$ a path in $\tree$ with endpoints  $\iota(u)$ and $\iota(v)$.
If $e\in E(G)$ and $h \in E(\tree)$ is on the path $\rho(e)$, then we say that ``$e$ is running parallel to $h$''.
The congestion of $G$ with respect to an embedding $\mathcal{E}$ of $G$ into a host graph $\tree$, denoted as $\operatorname{cng}_{\tree,\mathcal{E}}(G)$, is defined to be the maximal number of times an arc of $\tree$ is used in the routing of arcs of~$G$.
We also say that $\tree$ is {\em realizing} congestion $\operatorname{cng}_{\tree,\mathcal{E}}(G)$.
Several notions of congestion can be obtained by minimizing $\operatorname{cng}_{\tree,\mathcal{E}}(G)$ over various families of host graphs and embeddings (see, e.g., \cite{ostrovskii2004minimal}). Here we work with the following.

\begin{definition}[Congestion\footnote{It is important to note that congestion in the sense of Definition \ref{defn:cng} is also known as \emph{carving width}, a term which was coined by Robertson and Seymour in \cite{seymour1994call}. However, the usual abbreviation for carving width is `cw' which clashes with that of the cutwidth. Therefore we stick to the name `congestion' and the abbreviation `cng' to avoid this potential confusion in notation.}] 
Let $T_{\{1,3\}}$ be the set of unrooted binary trees.\footnote{An \textit{unrooted binary tree} is a tree in which each node is incident to either one or three arcs.} The {\em congestion} $\cng{G}$ of a graph $G$ is defined as
\[
	\cng{G}=\min\{\operatorname{cng}_{\tree,\mathcal{E}}(G):\tree\in T_{\{1,3\}}, ~\mathcal{E}=(\iota,\rho)~\text{with}~\iota\colon V(G)\rightarrow L(\tree)~\text{bijection}\},
\]
where $L(\tree)$ denotes the set of leaves of $\tree$.

In other words, we minimize $\operatorname{cng}_{\tree,\mathcal{E}}(G)$ when the host graph $\tree$ is an unrooted binary tree and the mapping $\iota$ maps the nodes of $G$ bijectively onto the leaves of $\tree$. The routing $\rho$ is uniquely determined as the host graph is a tree. See Figure \ref{fig:routing}.
\label{defn:cng}
\end{definition} 

\begin{figure}[ht]
	\centerline{\includegraphics[scale=1.2]{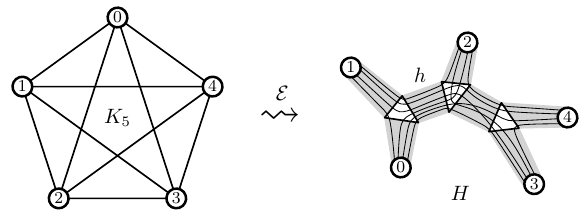}}
	\caption{The complete graph $K_5$ (guest) routed along an unrooted binary tree $\tree$ (host). We have $\operatorname{cng}_{\tree,\mathcal{E}}(K_5)=6$ which is witnessed by $h$: six arcs of $K_5$ are running parallel to $h$}
	\label{fig:routing} 
\end{figure}

\begin{theorem}[Bienstock, p.\ 108--111 of \cite{bienstock1990embedding}] Given a graph $G$ with maximum degree $\Delta(G)$, we have\footnote{Only the right-hand side inequality of Theorem \ref{thm:tw-cng-tw}, $\cng{G} \leq (\tw{G}+1)\Delta(G)$, is  formulated explicitly in \cite{bienstock1990embedding} as Theorem 1 on p.\ 111, whereas the left-hand side inequality is stated ``inline'' in the preceding paragraphs on the same page.}
\[
	\max\left\{\textstyle \frac{2}{3}(\tw{G}+1),\Delta(G)\right\} \leq \cng{G} \leq \Delta(G)(\tw{G}+1).
\]
\label{thm:tw-cng-tw}
\end{theorem}

\begin{example}[The Petersen graph]
\label{ex:petersen}
One of the most widely used examples in graph theory is the \textit{Petersen graph}, denoted $\petersen$, see Figure \ref{fig:petersen}(i). Although it is not a dual graph of a $3$-manifold triangulation (since it is not 4-regular), it turns out to be helpful for comparing the graph parameters considered in this article.

\begin{itemize}
	\item $\boldsymbol{\cw{\petersen}=6}$. Notice that for any $S \subset V=V(\petersen)$ of cardinality four there are at least six arcs running between $S$ and $V\setminus S$. That is, on the one hand, $\cw{\petersen} \geq 6$. On the other hand, it is easily verified that in the linear ordering $0 < 1 < 2 < \ldots < 9$ the maximal cutset has size six.
	\item $\boldsymbol{\pw{\petersen}=5}$. A minimal-width path decomposition (which we computed using the module `Vertex separation' of {\em SageMath} \cite{sage}) is the following.
	
	\medskip

	\centerline{\includegraphics[scale=1.09]{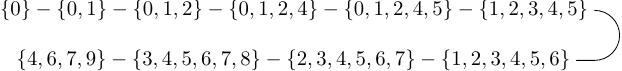}}

	\item $\boldsymbol{\tw{\petersen}=4}$. An optimal tree decomposition (computed using {\em SageMath} \cite{sage}) is shown in Figure \ref{fig:petersen}(ii).
	\item $\boldsymbol{\operatorname{cng}(\petersen)=5}$. Every arc $e$ of a host tree $\tree$ specifies a cut in $\petersen$: by deleting $e$, the leaves of the two components of $\tree\setminus e$ correspond to a partition of $V(\petersen)$.
It is easy to see that there is always a cut $S \cup R = V$ with $\{\# S, \# R\} \in \{\{3,7\},\{4,6\},\{5,5\}\}$, and that every such cut contains at least five arcs, hence $\operatorname{cng}(\petersen)\geq 5$. The reverse inequality is proven through the example in Figure \ref{fig:petersen}(iii).
\end{itemize}
\end{example}

\begin{figure}[ht]
	\centerline{\includegraphics[scale=1.3]{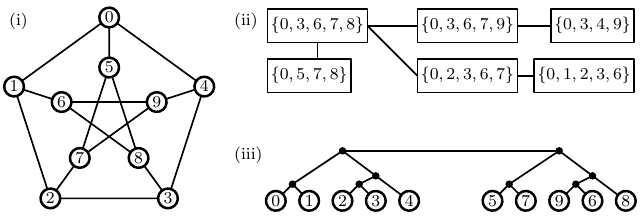}}
	\caption{The Petersen graph (i), a tree decomposition realizing minimal treewidth (ii), and an unrooted binary tree realizing minimal congestion (iii)}
	\label{fig:petersen}
\end{figure}

\section{Width-type parameters for 3-manifolds}
\label{sec:3mfdParams}

The Heegaard genus (defined in Section \ref{ssec:3mfds}) is a first example for a width-type parameter of a 3-manifold: the larger the Heegaard genus, the ``wider'' the manifold.
Here we consider two subsequent refinements, the linear width and the more general graph width, whose properties are essential for proving our results in Sections \ref{sec:lwpw} and \ref{sec:toptw}, respectively.

\paragraph*{Linear width.} In \cite{scharlemann1992thin} Scharlemann and Thompson extend the concept of thin position from knot theory \cite{gabai1987thinpositon} to $3$-manifolds and define the linear width of a manifold.\footnote{Also see \cite{hugh2007thinposition} and the textbook \cite{scharlemann2016lecture} for an introduction to generalized Heegaard splittings and to thin position, and for a survey of recent results.}
For this they look at linear splittings, i.e., linear decompositions of a manifold into compression bodies. This setup is explained in Section~\ref{ssec:3mfds}.

Given a linear splitting of a $3$-manifold $\manifold$ into $2s$ compression bodies with top boundary surfaces $\altaltsurface_i$, $ 1 \leq i \leq s$, consider the multiset $\{ c(\altaltsurface_i) : 1 \leq i \leq s \}$, where $c(\surface) = \max \{ 0, 2g(\surface) - 1 \}$ for a connected surface $\surface$, and $c(\surface) = \sum_j c(S_j)$ for a surface $\surface$ with connected components $S_j$. This multiset $\{ c(\altaltsurface_i) : 1 \leq i \leq s \}$, when arranged in non-increasing order, is called the {\em width} of the (linear) splitting. We here define the {\em linear width of a manifold} $\manifold$, denoted by $\CW{\manifold}$, to be the maximum entry in a lexicographically smallest width ranging over all linear splittings of $\manifold$.\footnote{For our purposes it is most convenient to define the linear width to be a single integer rather than a multiset of integers. We thus deviate at this point from the definition of linear width in \cite{scharlemann1992thin}.} A manifold $\manifold$ together with a linear spitting of lexicographically smallest width is said to be in {\em thin position}.

A guiding idea behind thin position is to attach $2$- and $3$-handles as early as possible and $0$- and $1$-handles as late as possible in order to obtain a decomposition for which the ``topological complexity'' of the top bounding surfaces is minimized.

\begin{theorem}[Scharlemann--Thompson, Rule 5 from \cite{scharlemann1992thin}]
\label{thm:incompressible}
Let $\manifold$ be a $3$-manifold together with a linear splitting into compression bodies $(\compbodyone_1, \compbodytwo_1, \ldots , \compbodyone_s, \compbodytwo_s)$ in thin position, and let $\{\altsurface_i \subset \manifold : 1 \leq i \leq s-1\}$, be the set of lower boundary surfaces $\altsurface_i = \partial_-\compbodytwo_i = \partial_-\compbodyone_{i+1}$. Then every connected component of every surface $\altsurface_i$ is incompressible.
\end{theorem}

Theorem~\ref{thm:incompressible} has the following consequence.

\begin{theorem}[Scharlemann--Thompson \cite{scharlemann1992thin}]
\label{thm:st}
Let $\manifold$ be irreducible, non-Haken. Then the smallest width linear splitting of $\manifold$ into compression bodies 
is a Heegaard splitting of minimal genus $\mathfrak{g}(\manifold)$. In particular, the 
linear width of $\manifold$ is given by $\CW{\manifold}=2\mathfrak{g}(\manifold)-1$.
\end{theorem}

\begin{proof}[Proof (sketch)]
Let $(\compbodyone_1, \compbodytwo_1, \compbodyone_2, \compbodytwo_2, \ldots , \compbodyone_s, \compbodytwo_s)$ be a thin decomposition of $\manifold$. By Theorem~\ref{thm:incompressible}, all surface components of all bounding surfaces $\altsurface_i$, $1 \leq i < s$, must be incompressible. Moreover, all $\altsurface_i$ are separating and thus they are $2$-sided. However, irreducible, non-Haken $3$-manifolds do not have incompressible $2$-sided surfaces. Hence $s=1$ and therefore the decomposition $(\compbodyone_1, \compbodytwo_1)$ must be a Heegaard splitting of $\manifold$. 
\end{proof}

\paragraph*{Graph width.} In \cite{scharlemann2016lecture} Scharlemann, Schultens and Saito further refine the concept of thin position to graph splittings of $3$-manifolds, see Section~\ref{ssec:3mfds}. In particular, given a manifold $\manifold$ together with a graph splitting defined by a fork complex $\forkcomp$, let $\{\altaltsurface_j:j~\text{grip~of}~\forkcomp\}$ be the set of top boundary surfaces of the decomposition. Then the width of the graph splitting coming from $\forkcomp$ is defined as the multiset $\{g(\altaltsurface_j):j~\text{grip~of}~\forkcomp\}$ with non-increasing order. Similar to the construction of linear width, the {\em graph width $\graph{\manifold}$} of $\manifold$ is defined to be the largest entry of the lexicographically smallest width ranging over all graph splittings of $\manifold$. A graph splitting of $\manifold$ which has lexicographically smallest width is said to be {\em thin}.

\begin{theorem}[Scharlemann--Schultens--Saito, \cite{scharlemann2016lecture}  Corollary 5.2.5]
\label{thm:graphincompressible}
Let $\manifold$ be a $3$-manifold together with a thin graph splitting defined by fork a complex $\forkcomp$, and let $\{\altsurface_i \subset \manifold : i~\text{tine~of}~\forkcomp\}$ be the set of lower boundary surfaces as defined in Section~\ref{ssec:3mfds}. Then every connected component of every lower boundary surface $\altsurface_i$ is incompressible.
\end{theorem}

Similarly to the linear width case, Theorem~\ref{thm:graphincompressible} implies that a thin graph splitting of an irreducible, non-Haken $3$-manifold must be a Heegaard splitting. In particular, $\graph{\manifold} = \mathfrak{g} (\manifold)$ for any given irreducible, non-Haken $3$-manifold $\manifold$.

\paragraph*{Non-Haken 3-manifolds of large genus.} The next theorem provides an infinite family of 3-manifolds for which we can apply our results established in the subsequent sections.

\begin{theorem}[Agol, Theorem 3.2 in \cite{agol2003small}]
\label{thm:a}
There exist orientable, closed, irreducible, and non-Haken $3$-manifolds of arbitrarily large Heegaard genus. 
\end{theorem}

\begin{remark}
The construction used to prove Theorem \ref{thm:a} starts with non-Haken $n$-component link complements, and performs Dehn fillings which neither create incompressible surfaces, nor decrease the (unbounded) Heegaard genera of the complements. The existence of such Dehn fillings is guaranteed by work due to Hatcher \cite{hatcher1982boundary} and Moriah--Rubinstein \cite{moriah1997heegaard}. As can be deduced from the construction, the manifolds in question are closed and orientable.
\end{remark}

\section{An obstruction to bounded cutwidth and pathwidth}
\label{sec:lwpw}

In this section we establish an upper bound for the Heegaard genus of a $3$-manifold $\manifold$ in terms of the pathwidth of any triangulation of $\manifold$ (cf. Theorem~\ref{thm:heeg_pw}). As an application of this bound we prove Theorem~\ref{thm:pw}. That is, we show that there exists an infinite family of $3$-manifolds not admitting triangulations of uniformly bounded pathwidth.

\begin{theorem}
\label{thm:width}
Let $\manifold$ be a closed, orientable $3$-manifold of linear width $\CW{\manifold}$. Furthermore, let $\tri$ be a triangulation of $\manifold$ with dual graph $\Gamma (\tri)$ of cutwidth $\cw{\Gamma(\tri)}$.
Then we have $\CW{\manifold} \leq 6\cw{\Gamma(\tri)}+7$.
\end{theorem}

\begin{proof}[Proof of Theorem \ref{thm:heeg_pw} assuming Theorem~\ref{thm:width}]
By Theorem \ref{thm:parameters}, $\cw{\Gamma(\tri)} \leq 4\pw{\Gamma(\tri)}$ since dual graphs of $3$-manifold triangulations are 4-regular. By Theorem~\ref{thm:st}, $\CW{\manifold} = 2\mathfrak{g}(\manifold)-1$ whenever $\manifold$ is irreducible and non-Haken. Combining these relations with the inequality provided by Theorem~\ref{thm:width} yields the result.
\end{proof}

Theorem~\ref{thm:pw} is now obtained from Theorem~\ref{thm:heeg_pw} and Agol's Theorem~\ref{thm:a}.
It remains to prove Theorem \ref{thm:width}. 
We begin with a basic, yet very useful definition.

\newpage

\begin{definition}
\label{def:chd}
Let $\tri$ be a triangulation of a $3$-manifold $\manifold$. The {\em canonical handle decomposition $\chd{\tri}$} of $\manifold$ associated with $\tri$ is given by 
\begin{itemize}
\itemsep0em
    \item one $0$-handle for the interior of each tetrahedron of $\tri$,
    \item one $1$-handle for a thickened version of the interior of each triangle of $\tri$,
    \item one $2$-handle for a thickened version of the interior of each edge of $\tri$, and 
    \item one $3$-handle for a neighborhood of each vertex of $\tri$.
  \end{itemize}
\end{definition}

\begin{remark}
  In the above definition of a canonical handle decomposition of a triangulation we associate $0$-handles with tetrahedra and $3$-handles with vertices of the triangulation. This is motivated by the fact that we model this decomposition on the dual graph rather than on the triangulation itself. The reason for this choice, in turn, is that it is the dual graph of a triangulation which acts as an intermediary between the topology of a $3$-manifold and the framework of structural graph theory.
\end{remark}

The following lemma gives a bound on the complexity of boundary surfaces occurring in the process of building up a manifold $\manifold$ from the handles of the canonical handle decomposition of a given triangulation of $\manifold$.

\begin{lemma}
\label{lem:layoutToHandDec}
Let $\tri$ be a triangulation of a $3$-manifold $\manifold$ and let $\Delta_1 < \Delta_2 < \ldots < \Delta_n \in \tri$ be a linear ordering of its tetrahedra. Moreover, let $\hbody_1 \subset \hbody_2 \subset \ldots \subset \hbody_n = \chd{\tri}$ be a filtration of $\chd{\tri}$ where $ \hbody_j\subset \chd{\tri}$ is the codimension zero submanifold consisting of all handles of $\chd{\tri}$ disjoint from tetrahedra $\Delta_i$, $i > j$. Then passing from $\hbody_{j}$ to $\hbody_{j+1}$ corresponds to adding at most $15$ handles. Let $\compbodytwo$ be the codimension zero submanifold constructed from $\hbody_{j}$ by adding an arbitrary subset of these handles, then the sum of the genera of $\partial \compbodytwo$ is no larger than the sum of the genera of the components of $\partial \hbody_j$ plus four.
\end{lemma}

\begin{proof}[Proof of Lemma~\ref{lem:layoutToHandDec}]
This is apparent from the fact that every tetrahedron consists of $15$ (non-empty) faces and thus at most $15$ handles of $\chd{\tri}$ are disjoint from $\hbody_{j}$ but not disjoint from $\hbody_{j+1}$. In particular, at most $15$ handles are added at the $j$-th level of the filtration. Moreover, note that at most four of the handles added in every step are $1$-handles (corresponding to the four triangles of the tetrahedron), which are the only handles increasing the sum of the genera of $\partial \hbody_{j}$.
\end{proof}

With the help of Lemma~\ref{lem:layoutToHandDec} we can now prove Theorem \ref{thm:width}.

\begin{proof}[Proof of Theorem \ref{thm:width}]
Let $v_j$, $1 \leq j \leq n$, be the nodes of $\Gamma(\tri)$ with corresponding tetrahedra $\Delta_j \in \tri$, $1 \leq j \leq n$. We may assume, without loss of generality, that the largest cutset of the linear ordering $v_1 < v_2 < \ldots < v_n$ has cardinality~$\cw{\Gamma(\tri)} = k$. 

Let $\hbody_j \subset \chd{\tri}$, $1\leq j \leq n$, be the filtration from Lemma~\ref{lem:layoutToHandDec}. Moreover, let $C_j$, $1 \leq j < n$, be the cutsets of the linear ordering above. Naturally, the cutset $C_j$ can be associated with at most $k$ triangles of $\tri$ with, together, at most $3k$ edges and at most $3k$ vertices of $\tri$. Let $\hbody (C_j) \subset \chd{\tri}$ be the corresponding submanifold formed from the at most $k$ $1$-handles and at most $3k$ $2$- and $3$-handles each of $\chd{\tri}$ associated with these faces of~$\tri$. 

By construction, the boundary $\partial \hbody (C_j)$ of $\hbody (C_j)$ decomposes into two parts, one of which coincides with the boundary surface $\partial \hbody_j$.  Since $\hbody (C_j)$ is of the form ``neighborhood of $k$ triangles in $\tri$'', and since the $2$- and $3$-handles of $\chd{\tri}$ form a handlebody, the $2$- and $3$-handles of $\hbody (C_j)$ form a union of handlebodies with sum of genera at most $3k$.

To complete the construction of $\hbody (C_j)$, the remaining at most $k$ $1$-handles are attached to this union of handlebodies as $2$-handles. These either increase the number of boundary surface components, or decrease the overall sum of genera of the boundary components. Altogether, the sum of genera of $\partial \hbody_j \subset \partial \hbody (C_j)$ is bounded above by $3k$.
	
Hence, following Lemma~\ref{lem:layoutToHandDec}, the sum $g$ of genera of the components of any bounding surface for any sequence of handle attachments of $\chd{\tri}$ compatible with the ordering $v_1 < v_2 < \ldots < v_n$ is bounded above by $ 3\cw{\Gamma(\tri)} + 4$. It follows that $2g-1 \leq 6\cw{\Gamma(\tri)} + 7$, and finally, by the definition, $\CW{\manifold} \leq 6\cw{\Gamma(\tri)} + 7$.
\end{proof}

\section{An obstruction to bounded congestion and treewidth}
\label{sec:toptw}

The goal of this section is to prove Theorems~\ref{thm:tw} and \ref{thm:heeg_tw}, the counterparts of Theorem~\ref{thm:pw} and \ref{thm:heeg_pw} for treewidth.
At the core of the proof is the following explicit connection between the congestion of the dual graph of any triangulation of a $3$-manifold $\manifold$ and its graph width.

\begin{theorem}
\label{thm:graph_width_cng}
Let $\manifold$ be a closed, connected, orientable $3$-manifold of graph width $\graph{\manifold}$, and $\tri$ be a triangulation of $\manifold$ with dual graph $\Gamma (\tri)$ of congestion $\cng{\Gamma(\tri)}$.
Then either $\manifold$ has graph width $\graph{\manifold} \leq \frac{9}{2}\cng{\Gamma(\tri)}$, or $\tri$ only contains one tetrahedron.
\end{theorem}

\begin{proof}[Proof of Theorem \ref{thm:heeg_tw} assuming Theorem~\ref{thm:graph_width_cng}] First note that the only closed orientable $3$-manifolds which can be triangulated with a single tetrahedron are the $3$-sphere of Heegaard genus zero, and the lens spaces of type $\operatorname{L}(4,1)$ and $\operatorname{L}(5,2)$ of Heegaard genus one (see Example~\ref{ex:lens}) for which Theorem \ref{thm:heeg_tw} holds.
Otherwise, since dual graphs of $3$-manifold triangulations are $4$-regular, Theorem \ref{thm:tw-cng-tw} yields $\cng{\Gamma(\tri)} \leq 4(\tw{\Gamma(\tri)}+1)$. In addition, Theorem~\ref{thm:graphincompressible} implies that $\graph{\manifold} = \mathfrak{g}(\manifold)$ whenever $\manifold$ is irreducible and non-Haken. Combining these relations with the inequality provided by Theorem~\ref{thm:graph_width_cng} we obtain Theorem \ref{thm:heeg_tw}.
\end{proof}

Similarly as in Section \ref{sec:lwpw}, Theorem \ref{thm:tw} immediately follows from Theorems \ref{thm:heeg_tw} and \ref{thm:a}.
Hence, the remainder of the section is dedicated to the proof of Theorem \ref{thm:graph_width_cng}.

\begin{proof}[Proof of Theorem \ref{thm:graph_width_cng}] 
Let $\manifold$ be a closed, connected, orientable $3$-manifold, $\tri$ be a triangulation of $\manifold$ whose dual graph $\Gamma(\tri)$ has congestion $\cng{\Gamma (\tri)} \leq k$, and let $\tree$ be an unrooted binary tree realizing $\cng{\Gamma (\tri)} \leq k$ (cf.\ Definition \ref{defn:cng}).

If $k=0$, $\tri$ must consist of a single tetrahedron, and the theorem holds. Thus we can assume that~$k \geq 1$. 

The idea of the proof is to first construct a graph splitting of $\manifold$ from a fork complex $\forkcomp$ modeled on $\tree$ (cf.\ Section \ref{ssec:3mfds}), and then to analyze the genera of the top boundary surfaces appearing in the splitting to see that they are all bounded above by $\frac{9}{2}\cng{\Gamma(\tri)}$.

\begin{figure}
  \centerline{\includegraphics[scale=1.15]{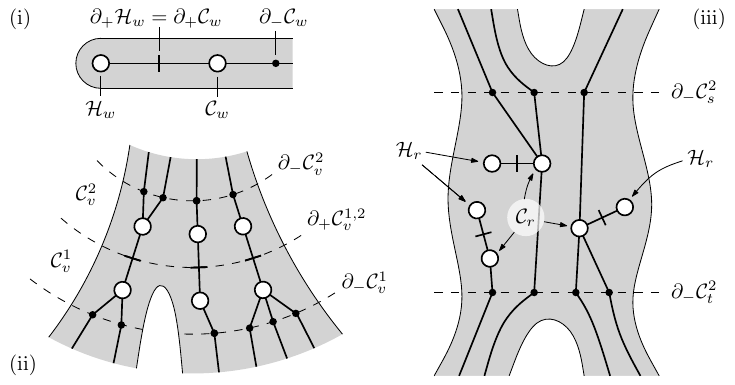}}
  \caption{Local pictures of the fork complex $\forkcomp$ constructed in the proof of Theorem \ref{thm:graph_width_cng}}
  \label{fig:forkcomplex}
\end{figure}

\paragraph*{Construction of the splitting.} Consider the canonical handle decomposition $\chd{\tri}$ of $\manifold$ associated with $\tri$ as defined in Definition \ref{def:chd}. Every compression body in the graph splitting described below is either a union of handles in $\chd{\tri}$, a thickened surface parallel to the boundary surface of some union of handles from $\chd{\tri}$, or a combination of both. In particular, the graph splitting maintains the handle structure coming from $\chd{\tri}$.
Note that we do {\em not} require the following compression bodies to be connected, but rather to be the union of connected compression bodies.

For every leaf $w \in V(\tree)$, a handlebody $\mathcal{H}_w$ is constructed as follows. Consider the abstract tetrahedron $\Delta_w \in \tri$ associated to $w$. If $\Delta_w$ has no self-identifications in $\tri$, then $\mathcal{H}_w$ is just the $0$-handle of $\chd{\tri}$ corresponding to $\Delta_w$. If $\Delta_w$ exhibits self-identifications, then first note that at most one pair of triangular faces of $\Delta_w$ are identified, otherwise $\Delta_w$ would be disjoint from the rest of $\tri$. Up to symmetry there are two possibilities:

%Either, $\Delta_w$ forms a ``snapped'' $3$-ball in $\tri$, see Figure \ref{fig:hbody_w}(i), in which case $\hbody_w$ is built \unskip\parfillskip 0pt \par}
	
\noindent
\begin{minipage}{0.5\textwidth}
Either, $\Delta_w$ forms a ``snapped'' $3$-ball in $\tri$, see Figure \ref{fig:hbody_w}(i), in which case $\hbody_w$ is built from the $0$-handle, the $1$-handle, and the $2$-handle of $\chd{\tri}$ corresponding to $\Delta_w$, to the face gluing, and to the edge $\{1,2\}$ of $\Delta_w$, respectively.

\hspace*{1em} Or,  $\Delta_w$ forms a solid torus in $\tri$, see Figure \ref{fig:hbody_w}(ii), and then $\hbody_w$ consists of the $0$-handle and of the $1$-handle of $\chd{\tri}$, corresponding to $\Delta_w$ and to the face gluing, respectively.
\end{minipage}\hfill
\begin{minipage}{0.46\textwidth}
  \centering
  \includegraphics[scale=.8]{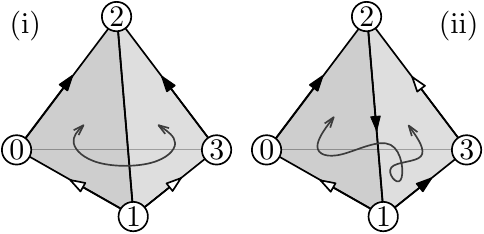}
  \captionof{figure}{(i) A snapped 3-ball, and\\(ii) a one-tetrahedron solid torus}
  \label{fig:hbody_w}
\end{minipage}~

\medskip

Moreover, for every leaf $w \in V(\tree)$, a compression body $\mathcal{C}_w = \partial \mathcal{H}_w \times [0,1]$ is attached to $\partial\mathcal{H}_w$ along $\partial_+\mathcal{C}_w = \partial\mathcal{H}_w \times \{0\}$. See Figure \ref{fig:forkcomplex}(i).

\medskip

Before proceeding, let us fix a ``root arc'' $r \in E(\tree)$.
This choice induces a partial order on $V(\tree)$: for $x,y \in V(\tree)$, $x \prec y$ if and only if $y$ is contained by the path connecting $x$ with $r$. We also say ``$x$ is below $y$''. In particular, $x \prec x$ for all $x \in V(\tree)$. Given $x \in V(\tree)$, let $\tri_x$ denote the submanifold of $\manifold$ consisting of
\begin{itemize}
\itemsep0em
	\item any $0$-handle of $\chd{\tri}$ corresponding to a leaf of $\tree$ below $x$, 
	\item any $1$-handle of $\chd{\tri}$ where both adjacent $0$-handles are in $\tri_x$
	\item any $2$-handle of $\chd{\tri}$ where all adjacent $0$-handles are in $\tri_x$, and
	\item any $3$-handle of $\chd{\tri}$ where all adjacent $0$-handles are in $\tri_x$.
\end{itemize}

\noindent
In other words, $\tri_x$ is the submanifold of $\manifold$ {\em spanned by} the $0$-handles of $\chd{\tri}$ below $x$.

\bigskip

\noindent
{\bf Claim 1.} If $x \prec y$ then $\tri_x \subseteq \tri_y$. If $x \not \prec y$ and $y \not \prec x$ then $\tri_x \cap \tri_y = \emptyset $ and $\partial \tri_x \cap \partial \tri_y = \emptyset$.

\medskip

\noindent
{\itshape Proof of Claim 1.} The first part of the claim is obvious.
For the second part, let $x,y  \in V(\tree)$ with $x \not \prec y$ and $y \not \prec x$. The way we construct $\tri_x$ and $\tri_y$ ensures that $\tri_x$ and $\tri_y$ do not only have disjoint interiors, but are also separated from each other such that their boundaries are disjoint as well.
Indeed, if $\tri_{x}$ and $\tri_{y}$ are both collections of $0$-handles this is certainly true as each of these $0$-handles can be thought of as living inside a single tetrahedron of $\tri$ away from its boundary.
As soon as we have two such $0$-handles living in two adjacent tetrahedra in, say, $\tri_x$, the $1$-handle(s) corresponding to their common triangular face(s) can be glued to the two $0$-handles and the resulting boundary is still disjoint from any $\tri_{y}$ being such a collection of $0$- and $1$-handles itself.
Now fix an edge $f$ of $\tri$ and suppose that $\tri_x$ contains all $0$- and $1$-handles associated to tetrahedra and triangles around $f$, then we are safe to add the $2$-handle corresponding to $f$ to $\tri_x$ and still be disjoint from $\tri_y$ even if $\tri_y$ itself is such a collection of $0$-, $1$-, and $2$-handles. 
The case of adding $3$-handles to $\tri_x$ and $\tri_y$ corresponding to vertices of $\tri$ is completely analogous to the previous case.
This proves the second part of the claim.

\bigskip
			
We can now describe the remaining parts of the graph splitting of $\manifold$. For this, let $v \in V(\tree)$ be a degree three node, $u,u' \in V(\tree)$ be the two nodes below and incident to $v$, and let $e, e' \in E(\tree)$ denote the arcs with endpoints $v$ and $u$, $u'$, respectively. For every such degree three node $v \in V(\tree)$, a pair of compression bodies $(\mathcal{C}^1_v,\mathcal{C}^2_v)$ is constructed in two steps. (See Figure \ref{fig:forkcomplex}(ii) for an example of a local schematic picture of $\forkcomp$ around $v$.)
		
\begin{enumerate}
	\item To construct the first compression body $\mathcal{C}^1_v$, we start with $(\partial (\tri_{u} \cup \tri_{u'})) \times [0,1]$ and attach to $(\partial (\tri_u \cup \tri_{u'})) \times \{1\}$ all $1$-handles of $\chd{\tri}$ corresponding to triangles of $\tri$ associated to arcs of $\Gamma (\tri)$ running parallel to $e$ {\em and} $e'$. (These are the remaining $1$-handles of $\tri_v$ not attached earlier.) We then define its lower and upper boundary as $\partial_{-} \mathcal{C}^1_v = (\partial (\tri_{u} \cup \tri_{u'})) \times \{0\}$ and $\partial_{+}\mathcal{C}^1_v = \partial \mathcal{C}^1_v \setminus \partial_{-} \mathcal{C}^1_v $, respectively.
	\item For the second compression body $\mathcal{C}^2_v$, we start with $\partial_{+}\mathcal{C}^1_v \times [0,1]$ (with the top boundary being defined as $\partial_{+} \mathcal{C}^2_v = \partial_{+} \mathcal{C}^1_v \times \{1\}$). The compression body is then completed by attaching along $\partial_{+}\mathcal{C}^1_v \times \{ 0 \}$ all $2$- and $3$-handles of $\chd{\tri}$ which are contained in $\tri_v$ but not in $\tri_{u} \cup \tri_{u'}$. We set $\partial_{-}\mathcal{C}^2_v = \partial \mathcal{C}^2_v \setminus \partial_{+} \mathcal{C}^1_v$.
\end{enumerate}

For the root arc $r = \{s,t\}$, we construct a pair of compression bodies $(\mathcal{C}_r,\mathcal{H}_r)$ as follows. ($\mathcal{H}_r$ is a union of handlebodies, hence the notation.)

\begin{enumerate}
	\item To build $\mathcal{C}_r$, start with $(\partial_{-}\mathcal{C}_s^2 \cup \partial_{-}\mathcal{C}_t^2) \times [0,1]$, define the lower boundary as $\partial_{-}\mathcal{C}_r = (\partial_{-}\mathcal{C}_s^2 \cup \partial_{-}\mathcal{C}_t^2) \times \{0\}$ and attach to $(\partial_{-}\mathcal{C}_s^2 \cup \partial_{-}\mathcal{C}_t^2) \times \{1\}$ all $1$-handles of $\chd{\tri}$ corresponding to arcs of $\Gamma(\tri)$ routed {\em through} $r$. As usual, $\partial_{+}\mathcal{C}_r = \partial \mathcal{C}_r \setminus \partial_{-} \mathcal{C}_r$.
	\item Finally, to obtain $\mathcal{H}_r$, take $\partial_{+}\mathcal{C}_r \times [0,1]$, set $\partial_{+}\mathcal{H}_r =\partial_{+}\mathcal{C}_r \times \{1\}$ and identify it with $\partial_{+}\mathcal{C}_r$, and attach all remaining $2$- and $3$-handles of $\chd{\tri}$ to $\partial_{+}\mathcal{C}_r \times \{0\}$.	
\end{enumerate}		
Figure \ref{fig:forkcomplex}(iii) shows a possible scenario around the root arc. This finishes the construction.

\bigskip

\noindent
{\bf Claim 2.} The compression bodies $\mathcal{H}_w$, $\mathcal{C}_w$, $\mathcal{C}^1_v$, $\mathcal{C}^2_v$, $\mathcal{C}_r$, and $\mathcal{H}_r$ (where $w,v \in V(\tree)$, $\deg(w)=1$, $\deg(v)=3$, and $r \in E(\tree)$ is the root arc), glued together along their appropriate boundary components, form a graph splitting of $\manifold$.

\medskip

\noindent
{\itshape Proof of Claim 2.} It follows from Claim 1 and the construction that all compression bodies above have pairwise disjoint interiors. We check that their lower and upper bonudary components match up whenever they are identified (cf.\ Figure~\ref{fig:forkcomplex}).  For the identifications between $\mathcal{H}_w$ and $\mathcal{C}_w$, between $\mathcal{C}_v^1$ and $\mathcal{C}_v^2$, and between $\mathcal{C}_r$ and $\mathcal{H}_r$ this is immediate. Now let $v \in V(\tree)$ be of degree three with adjacent nodes $u,u' \in V(\tree)$ below. Note that, by construction, $\partial_-\mathcal{C}_u^2=\partial\tri_u$ and $\partial_-\mathcal{C}_{u'}^2=\partial\tri_{u'}$ and they are disjoint by Claim 1. Hence $\partial_-\mathcal{C}_v^1 = (\partial(\tri_u \cup \tri_{u'}))\times\{0\}$ can indeed be identified with the disjoint union of $\partial_- \mathcal{C}_{u}^2$ and $\partial_- \mathcal{C}_{u'}^2$. For the gluings between $\partial_-\mathcal{C}_r$ and $\partial_- \mathcal{C}_s^2 \cup \partial_- \mathcal{C}_t^2$, where $r = \{s,t\} \in E(\tree)$ is the root arc, the reasoning is analogous. Finally, as it is modeled on the tree, the fork complex $\forkcomp$ is  {\em exact} (see \cite[Definition 5.1.4]{scharlemann2016lecture}), yielding a graph splitting of $\manifold$.

\paragraph*{Bounding the width.} Following \cite[Section 5.1]{scharlemann2016lecture}, cf. Section~\ref{sec:3mfdParams}, the width of the graph splitting of $\mathcal{M}$ exhibited above is given by the largest genus of a top boundary of a connected compression body in the graph splitting. However, this splitting, by construction, consists of unions of compression bodies. In particular, $\mathcal{C}^1_v, \mathcal{C}^2_v, \mathcal{C}_r$, and $\mathcal{H}_r$ may be disconnected. Note that this is not a problem since the sum of genera of top boundaries for every such union cannot be smaller than the genus of the largest genus compression body in the union. Hence, with this adjustment, we are left with the multiset  
\[
\left\{ g(\partial \mathcal{H}_w), g( \partial_{+} \mathcal{C}^1_v), g(\partial\mathcal{H}_r) ~\Big|~ w,v \in V(\tree), \deg(w)=1, \deg(v)=3, r \in E(\tree)~\text{root~arc} \right\}
\]
to determine an upper bound on the graph width of $\mathcal{M}$, where $g(\mathcal{S})$ denotes the sum of the genera of all connected components of $\mathcal{S}$.

The handlebody $\mathcal{H}_w$ has genus at most one: $g(\partial\mathcal{H}_w)=0$ if $\Delta_w$ is a $0$-handle, or forms a ``snapped'' $3$-ball in $\tri$, and $g(\partial\mathcal{H}_w)=1$ if $\Delta_w$ forms a solid torus in $\tri$.

Let us fix a node $v \in V(\tree)$ of degree three. Our goal is to upper-bound $g(\partial_{+} \mathcal{C}^1_v)$. Note that, since $\cng{\Gamma (\tri)} \leq k$, at most $k$ arcs of $\Gamma (\tri)$ run parallel to each arc of $\tree$. Moreover, counting those arcs of $\Gamma (\tri)$ along the three arcs of $\tree$ incident to $v$, we encounter each of them twice, therefore at most $\frac{3}{2} k$ arcs of $\Gamma (\tri)$ meet $v$. Based on this fact, we show that $g(\partial_{+} \mathcal{C}^1_v) \leq \frac{9k}{2}$. The proof relies on the next key observation. 

\medskip

\noindent
{\bf Claim 3.} Let $x \in V(H)$ be a node of $H$ and $a \in E(H)$ be the unique arc of $H$ above and incident to $x$. Then any handle $h \in \chd{\tri} \setminus \tri_x$ that touches $\partial\tri_x$ is adjacent or equal to a $1$-handle of $\chd{\tri}$ corresponding to an arc of $\Gamma(\tri)$ routed through $a$.

\smallskip

\noindent{\itshape Proof of Claim 3.}  Recall that $\tri_x$ is spanned by those $0$-handles of $\chd{\tri}$ that correspond to the leaves of $H$ below $x$. Turning this around, every handle in $\chd{\tri} \setminus \tri_x$ is either a $0$-handle whose corresponding leaf is {\em not} below $x$, or is adjacent to at least one such $0$-handle. Now let $h \in \chd{\tri} \setminus \tri_x$ be a handle that touches $\partial\tri_x$.

First, observe that $h$ cannot be a $0$-handle. Indeed, if $h$ is a $0$-handle not in $\tri_x$ then its corresponding leaf is not below $x$, and therefore all $h' \in \chd{\tri}$ adjacent to $h$ are not part of $\tri_x$ either. As the union of these handles $h'$ comprise a neighborhood of $h$, it follows that $\partial h \cap \partial \tri_x = \emptyset$, contradicting the assumption that $h$ touches $\partial\tri_x$.

Second, notice that if $h$ is an $i$-handle ($i \in \{1,2,3\}$) and no $0$-handles adjacent to $h$ are below $x$, then $h$ is separated from $\tri_x$ by the union of the $h' \in \chd{\tri}\setminus\{h\}$ that are adjacent to at least one of these $0$-handles. Thus there exists a $0$-handle $h_1 \in \chd{\tri}$ adjacent to $h$ with corresponding leaf below $x$. Moreover, some other $0$-handle adjacent to $h$, say $h_2$, must be in $\chd{\tri}\setminus\tri_x$, since otherwise $h$ must be part of $\tri_x$.

If $h$ is a $1$-handle, then $h_1$ and $h_2$ are precisely the two $0$-handles adjacent to $h$, which thus corresponds to an arc of $\Gamma(\tri)$ routed through $a \in E(H)$ and we are done. If $h$ is a $2$- or $3$-handle, then there is an alternating sequence $h_1 = h^{(0)}, h^{(1)}, h^{(2)}, \ldots, h^{(p)} = h_2$ of $0$- and $1$-handles adjacent to $h$, and $h^{(j)}$ being adjacent to $h^{(j+1)}$ ($0 \leq j < p$). Since $h_1 \in \tri_x$ and $h_2 \in \chd{\tri}\setminus\tri_x$, there exists some even $q \in \{0,2,\ldots, p-2\}$ for which $h^{(q)}  \in \tri_x$ and $h^{(q+2)} \in \chd{\tri}\setminus\tri_x$. But then $h^{(q+1)}$ is a $1$-handle of $\chd{\tri}$ adjacent to $h$ that corresponds to an arc of $\Gamma(\tri)$ routed through $a$. This concludes the proof of Claim 3.

By construction, $\partial_+\mathcal{C}_v^1 \approx \partial(\tri_u \cup \tri_{u'} \cup \mathscr{H}_{e,e'})$, where $\mathscr{H}_{e,e'}$ consists of the $1$-handles in $\chd{\tri}$ corresponding to arcs of $\Gamma(\tri)$ running parallel to both $e = \{u,v\}$ and  $e' = \{u',v\}$. (Here $\mathcal{S}_1 \approx \mathcal{S}_2$ denotes that the surfaces $\mathcal{S}_1$ and $\mathcal{S}_2$ are parallel, and hence, in particular, of the same genus.)

Let $\mathcal{X}$ be the submanifold of $\manifold$ built from the handles in $\chd{\tri} \setminus (\tri_u \cup \tri_{u'})$ touching $\partial(\tri_u \cup \tri_{u'})$.
It follows from Claim 3, that each handle in $\mathcal{X}$ is either a $1$-handle routed through $e$  or  $e'$, or a $2$- or $3$-handle adjacent to such a $1$-handle.  In particular, $\mathcal{X}$ consists of at most $\frac{3k}{2}$ $1$-handles, at most $\frac{9k}{2}$ $2$-handles, and at most $\frac{9k}{2}$ $3$-handles, cf.\ the paragraph before Claim 3. Since the $2$- and $3$-handles of $\chd{\tri}$ form a handlebody, the $2$- and $3$-handles of $\mathcal{X}$ form a union of handlebodies, denoted by $\mathcal{X}_{2,3}$, with sum of genera at most $\frac{9k}{2}$.

Consider the submanifold $\mathcal{Y} \subseteq \mathcal{X}$ obtained from $\mathcal{X}_{2,3}$ by attaching to it all $1$-handles of $\mathcal{X}$ {\em not} in $\mathscr{H}_{e,e'}$. (These are precisely the $1$-handles of $\chd{\tri}$ that correspond to arcs of $\Gamma(\tri)$ running parallel either to $e$ or to $e'$ but not to both.) Note that these $1$-handles are attached to $\mathcal{X}_{2,3}$ as $2$-handles. Each of these attachments either increases the number of boundary surface components, or decreases the overall sum of genera of the boundary components by one. Consequently, $g(\partial\mathcal{Y}) \leq g(\partial\mathcal{X}_{2,3}) \leq \frac{9k}{2}$. Finally, by construction, $\partial_+\mathcal{C}_v^1$ is parallel to the union of some components of $\partial\mathcal{Y}$, and therefore $g(\partial_+\mathcal{C}_v^1) \leq g(\partial\mathcal{Y}) \leq \frac{9k}{2}$.

\medskip

Bounding above the genus of $\partial\mathcal{H}_r$ is analogous. The only difference is that $\partial\mathcal{H}_r \approx \partial(\tri_s \cup \tri_{t} \cup \mathscr{H}_{r})$, where $\mathscr{H}_{r}$ now consists of the at most $k$ $1$-handles in $\chd{\tri}$ which correspond to arcs of $\Gamma(\tri)$ routed through the root arc $r = \{s,t\} \in E(H)$. Here an even stronger bound holds, i.e., $g(\partial\mathcal{H}_r) \leq 3k < \frac{9k}{2}$.

\medskip

From the definition of graph width it immediately follows that $\graph{\manifold} \leq \frac{9k}{2}$. 
\end{proof}

\newpage

\appendix
\section{Most triangulations have large treewidth}
\label{app:most}

As mentioned in the \hyperref[sec:intro]{Introduction}, most triangulations of most $3$-manifolds must have large treewidth. Here we briefly review two well-known and simple observations that, together, imply this fact. Since they are hard to find in the literature, we also sketch their proofs.

\begin{prop}
\label{prop:manyMfds}
There exists a constant $C > 1$, such that there are at least $C^{n \log (n)}$ $3$-manifolds which can be triangulated with $\leq n$ tetrahedra.
\end{prop}

\begin{proof}[Sketch of the proof]
Note that the number of isomorphism classes of graphs is superexponential in the number of nodes. Consider the family of graph manifolds where the nodes are Seifert fibered spaces of constant size.\footnote{See \cite{orlik2006seifert} for an overview on graph manifolds and Seifert fiber spaces.} Conclude by the observation that graph manifolds modeled on non-isomorphic graphs are non-homeomorphic, cf.\ \cite[Section 3]{lackenby2017conditionally}.
\end{proof}

\begin{prop}
\label{prop:fewSmallTW}
Given $k \geq 0$, there exists a constant $C_k >1 $ such that there are at most $C_k^n$ triangulations of $3$-manifolds with dual graph of treewidth $\leq k$ and $\leq n$ tetrahedra.
\end{prop}

\begin{proof}[Sketch of the proof]
The property of a graph to be of bounded treewidth is closed under minors. Hence, it follows from a theorem of Norine, Seymour, Thomas and Wollan \cite{norine2006minorclosed} that the number of isomorphism classes of graphs with treewidth $\leq k$ is at most exponential in the number of nodes $n$ of the graph. Furthermore, any given graph can at most produce a number of combinatorially distinct triangulations exponential in~$n$.
\end{proof}

\begin{wrapfigure}{r}{0pt}
\centering
    \includegraphics[scale=.66]{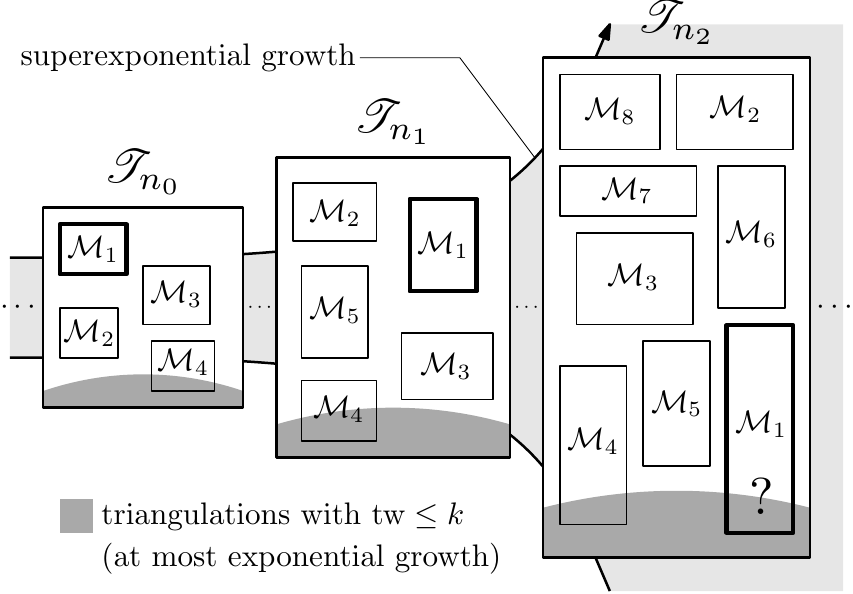}
	\caption{}
	\label{fig:manyMfds}
\end{wrapfigure}
\noindent These propositions tell us that, as $n$ grows larger, the number of triangulations of treewidth at most $k$ represent a rapidly decreasing fraction of the set $\mathscr{T}_n$ of all $(\leq n)$-tetrahedra triangulations of $3$-manifolds.

Let $\protect\boxed{\manifold_i} \subset \mathscr{T}_{n_j}$ denote the set of triangulations of the manifold $\manifold_i$ with at most $n_j$ tetrahedra. The main question we are investigating in this article is the following:

Given a $3$-manifold $\manifold$, does there always exist some $n_\manifold \in \mathbb{N}$, such that the set of triangulations $\protect\boxed{\manifold}$ in $\tri_{n_\manifold}$ overlaps with the region of treewidth $\leq k$ triangulations? In other words, does the set of triangulations of every $3$-manifold eventually behave like the one of $\manifold_1$ on Figure \ref{fig:manyMfds}?
Theorem~\ref{thm:tw} answers this question in the negative in general.

\newpage

\section{Complexity and fixed parameter tractability}
\label{app:complexity}

In Table \ref{tab:zoo} we collect complexity and fixed parameter tractability properties of the problems of computing the considered graph parameters. First, we explain the columns of the table.

\begin{itemize}
\itemsep0em 
	\item \textbf{Complexity.} The computational complexity of the question ``Is $\Par{G}\leq k$?''. Here $k$ is a variable given as part of the input.
	\item \textbf{FPT.} Fixed parameter tractability in the natural parameter. The check mark ($\checkmark$) indicates the following: if $k$ is fixed (as opposed to being a variable part of the input) and $G$ is an $n$-vertex graph, then the answer to the question ``Is $\Par{G} \leq k$?'' can be found in $O(\operatorname{poly}(n))$ time.
	\item \textbf{Bounded-degree graphs.} What is known if we restrict our attention to a family of bounded-degree graphs.
\end{itemize}

\vspace{-0.25em}

\begin{table}[ht]
\centering
\begin{tabular}{llll}  
\toprule
$p$ & \textbf{Complexity} & \textbf{FPT} & \textbf{Bounded-degree graphs}\\
\midrule
tw & ${\bf NP}$-complete \cite{arnborg1987complexity} & $\checkmark$ \cite{bodlaender2012fixed} & remains ${\bf NP}$-complete \cite{bodlaender1997treewidth}\\
pw & ${\bf NP}$-complete \cite{arnborg1987complexity} & $\checkmark$ \cite{bodlaender2012fixed} & remains ${\bf NP}$-complete \cite{monien1988min}\footnote{${\bf NP}$-completeness is shown for the vertex separation number which is equivalent to pathwidth \cite{kinnersley1992vertex}.}\\
cw & ${\bf NP}$-complete \cite{gavril1977some} & $\checkmark$ \cite{thilikos2005cutwidth} & polynomial if tw bounded \cite{thilikos2001polynomial}\\
cng & ${\bf NP}$-complete \cite{seymour1994call} & $\checkmark$ \cite{thilikos2000Carving}\footnote{See the discussion in the introduction of \cite{thilikos2000Carving}.} & \\
\bottomrule
\end{tabular}
\caption{Complexity and fixed parameter tractability of selected graph parameters}
\label{tab:zoo}
\end{table}
We point out that there is a more detailed table in \cite{bodlaender1994tourist}, showing the complexity of computing pathwidth and treewidth on several different classes of graphs.

\section{Computational aspects of different graph parameters}
\label{app:fpt}

A small treewidth $k \geq 0$ of the dual graph $\Gamma (\tri)$ of an $n$-tetrahedron triangulation $\tri$ of a $3$-manifold can be exploited by applying standard dynamic programming techniques to the tetrahedra of the triangulation: in a tree decomposition $\alttree$ of $\Gamma (\tri)$ with $O(n)$ bags realizing width $k$, every bag $B\in V(\alttree)$ corresponds to a subcomplex $X_B \subset \tri$ of at most $k+1$ tetrahedra of $\tri$. Going up from the leaves of $\alttree$, for each bag $B \in V(\alttree)$, compute a list of candidate solutions of the given problem on $X_B \subset \tri$. When processing a new bag $B' \in V(\alttree)$, for all child bags $B_i \in V(\alttree)$, $1 \leq i \leq r$, their lists of candidate solutions (which are already computed) are used to validate or disqualify candidate solutions for $X_{B'}$. Due to property~\ref{twproptwo} of a tree decomposition (Definition~\ref{defn-treewidth}), every time a tetrahedron disappears from a bag while we go up from the leaves to the root of $\alttree$, it never reappears. This means that constraints for a global solution coming from such a ``forgotten'' tetrahedron are fully incorporated in the candidate solutions of the current bags. If for each bag the running time, as well as the length of the list of candidate solutions is a function in $k$, the procedure must have running time $O(n)$ for triangulations with dual graphs of constant treewidth.

Moreover, for every tree decomposition $\alttree$ with $O(n)$ bags, there exists a linear time procedure to preprocess $\alttree$ into a tree decomposition $\alttree'$ of $\Gamma (\tri)$, also with $O(n)$ bags, such that every bag is of one of three types: {\em introduce}, {\em forget}, or {\em join bag} \cite{kloks1994treewidth}. Such a {\em nice tree decomposition}\footnote{See the $3$-manifold software {\em Regina} \cite{regina} for a visualization of a tree decomposition, and a nice tree decomposition of the dual graph of any given triangulation.} has the advantage that only three distinct procedures are needed to process all the bags---causing such FPT-algorithms to be much simpler in structure: Introduce and forget bags are bags in $\alttree$ with only one child bag where a node is either added to or removed from the child bag to obtain the parent bag. Procedures to deal with these situations are often comparatively simple to implement and running times are often comparatively feasible. A join bag is a bag with two child bags such that parent bag and both child bags are identical. Depending on the problem to be solved, the procedure of a join bag can be more intricate, and running times are often orders of magnitude slower than in the other two cases. 

\paragraph*{Pathwidth vs.\ treewidth.} Since every (preprocessed) nice path decomposition is a nice tree decomposition without join bags, every FPT-algorithm in the treewidth of the input is also FPT in the pathwidth with the (often dominant) running time of the join bag removed. Thus, at least in certain circumstances, it can be beneficial to work with nice path decompositions---and thus with pathwidth as a parameter---instead of treewidth and its more complicated join bags in their nice tree decompositions. 

\bigskip

For small cutwidth and for small congestion, similar dynamic programming techniques can be applied to the cutsets of the respective linear or binary tree layouts of the nodes and arcs of the dual graph of a triangulation. Thus, in the following paragraph we compare these parameters to treewidth (and pathwidth), and point out some potential benefits from using them as alternative parameters.   

\paragraph*{Congestion vs.\ treewidth and cutwidth vs.\ pathwidth.} Parameterized algorithms using pathwidth or treewidth operate on {\em bags} containing elements corresponding to tetrahedra of the input triangulations. In contrast to this, parameterized algorithms using cutwidth or congestion operate on {\em cutsets} containing elements corresponding to triangles of the input triangulations. It follows that an algorithm operating on a tree decomposition of width $k$ must handle a $3$-dimensional subcomplex of the input triangulation made of up to $15(k+1)$ faces in one step. An algorithm operating on a tree layout of congestion $k$, however, only needs to consider a $2$-dimensional subcomplex of up to $7k$ faces of the input triangulation per step. Moreover, cutwidth and congestion are equivalent to pathwidth and treewidth respectively (for bounded degree graphs, up to a small constant factor, see Theorems~\ref{thm:parameters} and \ref{thm:tw-cng-tw}), and parameterized algorithms for $3$-manifolds are not just theoretical statements, but may give rise to practical tools outperforming current state-of-the-art algorithms (see, for example, \cite{burton2018algorithms}).
These observations suggest that, at least for some problems, parameterized algorithms using cutwidth or congestion of the dual graph have a chance to outperform similar algorithms operating on pathwidth or treewidth, respectively.

\newpage

\bibliographystyle{my_plainurl}
\bibliography{references}

\begin{thebibliography}{10}

\bibitem{agol2003small}
I.~Agol.
\newblock Small 3-manifolds of large genus.
\newblock {\em Geom. Dedicata}, 102:53--64, 2003.
\newblock \href {https://dx.doi.org/10.1023/B:GEOM.0000006584.85248.c5}
  {\path{doi:10.1023/B:GEOM.0000006584.85248.c5}}, \href
  {https://mathscinet.ams.org/mathscinet-getitem?mr=2026837}
  {\path{MR:2026837}}, \href {https://zbmath.org/?q=an:1039.57008}
  {\path{Zbl:1039.57008}}.

\bibitem{arnborg1987complexity}
S.~Arnborg, D.~G. Corneil, and A.~Proskurowski.
\newblock Complexity of finding embeddings in a {$k$}-tree.
\newblock {\em SIAM J. Algebr. Discrete Methods}, 8(2):277--284, 1987.
\newblock \href {https://dx.doi.org/10.1137/0608024}
  {\path{doi:10.1137/0608024}}, \href
  {https://mathscinet.ams.org/mathscinet-getitem?mr=881187} {\path{MR:881187}},
  \href {https://zbmath.org/?q=an:0611.05022} {\path{Zbl:0611.05022}}.

\bibitem{bessieres2010geometrisation}
L.~Bessi\`eres, G.~Besson, S.~Maillot, M.~Boileau, and J.~Porti.
\newblock {\em Geometrisation of $3$-Manifolds}, volume~13 of {\em EMS Tracts
  Math.}
\newblock Eur. Math. Soc. (EMS), Z\"urich, 2010.
\newblock \href {https://dx.doi.org/10.4171/082} {\path{doi:10.4171/082}},
  \href {https://mathscinet.ams.org/mathscinet-getitem?mr=2683385}
  {\path{MR:2683385}}, \href {https://zbmath.org/?q=an:1244.57003}
  {\path{Zbl:1244.57003}}.

\bibitem{bienstock1990embedding}
D.~Bienstock.
\newblock On embedding graphs in trees.
\newblock {\em J. Comb. Theory, Ser. {B}}, 49(1):103--136, 1990.
\newblock \href {https://dx.doi.org/10.1016/0095-8956(90)90066-9}
  {\path{doi:10.1016/0095-8956(90)90066-9}}, \href
  {https://mathscinet.ams.org/mathscinet-getitem?mr=1056822}
  {\path{MR:1056822}}, \href {https://zbmath.org/?q=an:0646.05025}
  {\path{Zbl:0646.05025}}.

\bibitem{bienstock1989graph}
D.~Bienstock.
\newblock Graph searching, path-width, tree-width and related problems.
\newblock In {\em Reliability Of Computer And Communication Networks}, volume~5
  of {\em {DIMACS} Ser. Discrete Math. Theor. Comput. Sci.}, pages 33--50.
  Amer. Math. Soc., Providence, RI, 1991.
\newblock \href {https://dx.doi.org/10.1090/dimacs/005/02}
  {\path{doi:10.1090/dimacs/005/02}}, \href
  {https://mathscinet.ams.org/mathscinet-getitem?mr=1119138}
  {\path{MR:1119138}}, \href {https://zbmath.org/?q=an:0777.05090}
  {\path{Zbl:0777.05090}}.

\bibitem{bodlaender1994tourist}
H.~L. Bodlaender.
\newblock A tourist guide through treewidth.
\newblock {\em Acta Cybern.}, 11(1-2):1--21, 1993.
\newblock URL:
  \url{https://cyber.bibl.u-szeged.hu/index.php/actcybern/article/view/3417},
  \href {https://mathscinet.ams.org/mathscinet-getitem?mr=1268488}
  {\path{MR:1268488}}, \href {https://zbmath.org/?q=an:0804.68101}
  {\path{Zbl:0804.68101}}.

\bibitem{bodlaender1998partial}
H.~L. Bodlaender.
\newblock A partial \emph{k}-arboretum of graphs with bounded treewidth.
\newblock {\em Theor. Comput. Sci.}, 209(1--2):1--45, 1998.
\newblock \href {https://dx.doi.org/10.1016/S0304-3975(97)00228-4}
  {\path{doi:10.1016/S0304-3975(97)00228-4}}, \href
  {https://mathscinet.ams.org/mathscinet-getitem?mr=1647486}
  {\path{MR:1647486}}, \href {https://zbmath.org/?q=an:0912.68148}
  {\path{Zbl:0912.68148}}.

\bibitem{bodlaender2005discovering}
H.~L. Bodlaender.
\newblock Discovering treewidth.
\newblock In {\em Proc. 31st Conf. Curr. Trends Theory Pract. Comput. Sci.
  ({SOFSEM} 2005)}, pages 1--16, 2005.
\newblock \href {https://dx.doi.org/10.1007/978-3-540-30577-4_1}
  {\path{doi:10.1007/978-3-540-30577-4_1}}, \href
  {https://zbmath.org/?q=an:1117.68451} {\path{Zbl:1117.68451}}.

\bibitem{bodlaender2012fixed}
H.~L. Bodlaender.
\newblock Fixed-parameter tractability of treewidth and pathwidth.
\newblock In {\em The Multivariate Algorithmic Revolution and Beyond - Essays
  Dedicated to Michael R. Fellows on the Occasion of His 60th Birthday}, volume
  7370 of {\em Lect. Notes Comput. Sci.}, pages 196--227. Springer, Berlin,
  Heidelberg, 2012.
\newblock \href {https://dx.doi.org/10.1007/978-3-642-30891-8_12}
  {\path{doi:10.1007/978-3-642-30891-8_12}}, \href
  {https://zbmath.org/?q=an:1358.68119} {\path{Zbl:1358.68119}}.

\bibitem{bodlaender2008combinatorial}
H.~L. Bodlaender and A.~M. C.~A. Koster.
\newblock Combinatorial optimization on graphs of bounded treewidth.
\newblock {\em Comput. J.}, 51(3):255--269, 2008.
\newblock \href {https://dx.doi.org/10.1093/comjnl/bxm037}
  {\path{doi:10.1093/comjnl/bxm037}}.

\bibitem{bodlaender1997treewidth}
H.~L. Bodlaender and D.~M. Thilikos.
\newblock Treewidth for graphs with small chordality.
\newblock {\em Discrete Appl. Math.}, 79(1--3):45--61, 1997.
\newblock \href {https://dx.doi.org/10.1016/S0166-218X(97)00031-0}
  {\path{doi:10.1016/S0166-218X(97)00031-0}}, \href
  {https://mathscinet.ams.org/mathscinet-getitem?mr=1478240}
  {\path{MR:1478240}}, \href {https://zbmath.org/?q=an:0895.68113}
  {\path{Zbl:0895.68113}}.

\bibitem{burton2011detecting}
B.~A. Burton.
\newblock Detecting genus in vertex links for the fast enumeration of
  3-manifold triangulations.
\newblock In {\em Proc. Int. Symp. Symb. Alg. Comput. ({ISSAC} 2011)}, pages
  59--66, 2011.
\newblock \href {https://dx.doi.org/10.1145/1993886.1993901}
  {\path{doi:10.1145/1993886.1993901}}, \href
  {https://mathscinet.ams.org/mathscinet-getitem?mr=2895195}
  {\path{MR:2895195}}, \href {https://zbmath.org/?q=an:1323.68537}
  {\path{Zbl:1323.68537}}.

\bibitem{regina}
B.~A. Burton, R.~Budney, W.~Pettersson, et~al.
\newblock Regina: Software for low-dimensional topology, 1999--2019.
\newblock Version 5.1.
\newblock URL: \url{https://regina-normal.github.io}.

\bibitem{burton2017courcelle}
B.~A. Burton and R.~G. Downey.
\newblock Courcelle's theorem for triangulations.
\newblock {\em J. Comb. Theory, Ser. {A}}, 146:264--294, 2017.
\newblock \href {https://dx.doi.org/10.1016/j.jcta.2016.10.001}
  {\path{doi:10.1016/j.jcta.2016.10.001}}, \href
  {https://mathscinet.ams.org/mathscinet-getitem?mr=3574232}
  {\path{MR:3574232}}, \href {https://zbmath.org/?q=an:1353.05122}
  {\path{Zbl:1353.05122}}.

\bibitem{MFOReports}
B.~A. Burton, H.~Edelsbrunner, J.~Erickson, and S.~Tillmann, editors.
\newblock {\em Computational Geometric and Algebraic Topology}, volume~12 of
  {\em Oberwolfach Rep.} EMS Publ. House, 2015.
\newblock \href {https://dx.doi.org/10.4171/OWR/2015/45}
  {\path{doi:10.4171/OWR/2015/45}}, \href {https://zbmath.org/?q=an:1380.00045}
  {\path{Zbl:1380.00045}}.

\bibitem{burton2016parameterized}
B.~A. Burton, T.~Lewiner, J.~Paix{\~{a}}o, and J.~Spreer.
\newblock Parameterized complexity of discrete {M}orse theory.
\newblock {\em {ACM} Trans. Math. Softw.}, 42(1):6:1--6:24, 2016.
\newblock \href {https://dx.doi.org/10.1145/2738034}
  {\path{doi:10.1145/2738034}}, \href
  {https://mathscinet.ams.org/mathscinet-getitem?mr=3472422}
  {\path{MR:3472422}}, \href {https://zbmath.org/?q=an:1347.68165}
  {\path{Zbl:1347.68165}}.

\bibitem{burton2018algorithms}
B.~A. Burton, C.~Maria, and J.~Spreer.
\newblock Algorithms and complexity for {Turaev--Viro} invariants.
\newblock {\em J. Appl. Comput. Topol.}, 2(1--2):33--53, 2018.
\newblock \href {https://dx.doi.org/10.1007/s41468-018-0016-2}
  {\path{doi:10.1007/s41468-018-0016-2}}, \href
  {https://mathscinet.ams.org/mathscinet-getitem?mr=3873178}
  {\path{MR:3873178}}, \href {https://zbmath.org/?q=an:07089248}
  {\path{Zbl:07089248}}.

\bibitem{burton2013complexity}
B.~A. Burton and J.~Spreer.
\newblock The complexity of detecting taut angle structures on triangulations.
\newblock In {\em Proc. 24th Annu. {ACM-SIAM} Symp. Discrete Algorithms ({SODA}
  2013)}, pages 168--183, 2013.
\newblock \href {https://dx.doi.org/10.1137/1.9781611973105.13}
  {\path{doi:10.1137/1.9781611973105.13}}, \href
  {https://mathscinet.ams.org/mathscinet-getitem?mr=3185388}
  {\path{MR:3185388}}, \href {https://zbmath.org/?q=an:1421.68161}
  {\path{Zbl:1421.68161}}.

\bibitem{chung1989graphs}
F.~R.~K. Chung and P.~D. Seymour.
\newblock Graphs with small bandwidth and cutwidth.
\newblock {\em Discrete Math.}, 75(1--3):113--119, 1989.
\newblock \href {https://dx.doi.org/10.1016/0012-365X(89)90083-6}
  {\path{doi:10.1016/0012-365X(89)90083-6}}, \href
  {https://mathscinet.ams.org/mathscinet-getitem?mr=1001391}
  {\path{MR:1001391}}, \href {https://zbmath.org/?q=an:0668.05039}
  {\path{Zbl:0668.05039}}.

\bibitem{courcelle1990monadic}
B.~Courcelle.
\newblock The monadic second-order logic of graphs. {I.} {R}ecognizable sets of
  finite graphs.
\newblock {\em Inf. Comput.}, 85(1):12--75, 1990.
\newblock \href {https://dx.doi.org/10.1016/0890-5401(90)90043-H}
  {\path{doi:10.1016/0890-5401(90)90043-H}}, \href
  {https://mathscinet.ams.org/mathscinet-getitem?mr=1042649}
  {\path{MR:1042649}}, \href {https://zbmath.org/?q=an:0722.03008}
  {\path{Zbl:0722.03008}}.

\bibitem{sage}
The~Sage Developers.
\newblock {\em {S}age {M}athematics {S}oftware {S}ystem ({V}ersion 7.6)}, 2017.
\newblock URL: \url{http://www.sagemath.org}, \href
  {https://dx.doi.org/10.5281/zenodo.820864}
  {\path{doi:10.5281/zenodo.820864}}.

\bibitem{diestel2017graph}
R.~Diestel.
\newblock {\em Graph Theory}, volume 173 of {\em Grad. Texts Math.}
\newblock Springer, Berlin, 5th edition, 2017.
\newblock \href {https://dx.doi.org/10.1007/978-3-662-53622-3}
  {\path{doi:10.1007/978-3-662-53622-3}}, \href
  {https://mathscinet.ams.org/mathscinet-getitem?mr=3644391}
  {\path{MR:3644391}}, \href {https://zbmath.org/?q=an:1375.05002}
  {\path{Zbl:1375.05002}}.

\bibitem{downey1999parameterized}
R.~G. Downey and M.~R. Fellows.
\newblock {\em Parameterized Complexity}.
\newblock Monogr. Comput. Sci. Springer-Verlag New York, 1999.
\newblock \href {https://dx.doi.org/10.1007/978-1-4612-0515-9}
  {\path{doi:10.1007/978-1-4612-0515-9}}, \href
  {https://mathscinet.ams.org/mathscinet-getitem?mr=1656112}
  {\path{MR:1656112}}, \href {https://zbmath.org/?q=an:0914.68076}
  {\path{Zbl:0914.68076}}.

\bibitem{downey2013fundamentals}
R.~G. Downey and M.~R. Fellows.
\newblock {\em Fundamentals of Parameterized Complexity}.
\newblock Texts Comput. Sci. Springer, London, 2013.
\newblock \href {https://dx.doi.org/10.1007/978-1-4471-5559-1}
  {\path{doi:10.1007/978-1-4471-5559-1}}, \href
  {https://mathscinet.ams.org/mathscinet-getitem?mr=3154461}
  {\path{MR:3154461}}, \href {https://zbmath.org/?q=an:1358.68006}
  {\path{Zbl:1358.68006}}.

\bibitem{gabai1987thinpositon}
D.~Gabai.
\newblock Foliations and the topology of {$3$}-manifolds. {III}.
\newblock {\em J. Differ. Geom.}, 26(3):479--536, 1987.
\newblock \href {https://dx.doi.org/10.4310/jdg/1214441488}
  {\path{doi:10.4310/jdg/1214441488}}, \href
  {https://mathscinet.ams.org/mathscinet-getitem?mr=910018} {\path{MR:910018}},
  \href {https://zbmath.org/?q=an:0639.57008} {\path{Zbl:0639.57008}}.

\bibitem{gavril1977some}
F.~Gavril.
\newblock Some {NP}-complete problems on graphs.
\newblock In {\em Proc. 1977 Conf. on Inf. Sci. Syst.}, pages 91--95. Johns
  Hopkins Univ., 1977.

\bibitem{haken1961normal}
W.~Haken.
\newblock Theorie der {N}ormalfl\"achen.
\newblock {\em Acta Math.}, 105:245--375, 1961.
\newblock \href {https://dx.doi.org/10.1007/BF02559591}
  {\path{doi:10.1007/BF02559591}}, \href
  {https://mathscinet.ams.org/mathscinet-getitem?mr=0141106}
  {\path{MR:0141106}}, \href {https://zbmath.org/?q=an:0100.19402}
  {\path{Zbl:0100.19402}}.

\bibitem{hass1999computational}
J.~Hass, J.~C. Lagarias, and N.~Pippenger.
\newblock The computational complexity of knot and link problems.
\newblock {\em J. {ACM}}, 46(2):185--211, 1999.
\newblock \href {https://dx.doi.org/10.1145/301970.301971}
  {\path{doi:10.1145/301970.301971}}, \href
  {https://mathscinet.ams.org/mathscinet-getitem?mr=1693203}
  {\path{MR:1693203}}, \href {https://zbmath.org/?q=an:1065.68667}
  {\path{Zbl:1065.68667}}.

\bibitem{hatcher1982boundary}
A.~E. Hatcher.
\newblock On the boundary curves of incompressible surfaces.
\newblock {\em Pacific J. Math.}, 99(2):373--377, 1982.
\newblock \href {https://dx.doi.org/10.2140/pjm.1982.99.373}
  {\path{doi:10.2140/pjm.1982.99.373}}, \href
  {https://mathscinet.ams.org/mathscinet-getitem?mr=658066} {\path{MR:658066}},
  \href {https://zbmath.org/?q=an:0502.57005} {\path{Zbl:0502.57005}}.

\bibitem{hlinveny2008width}
P.~Hlinen{\'{y}}, S.~Oum, D.~Seese, and G.~Gottlob.
\newblock Width parameters beyond tree-width and their applications.
\newblock {\em Comput. J.}, 51(3):326--362, 2008.
\newblock \href {https://dx.doi.org/10.1093/comjnl/bxm052}
  {\path{doi:10.1093/comjnl/bxm052}}.

\bibitem{hoffoss2016morse}
D.~Hoffoss and J.~Maher.
\newblock Morse area and {S}charlemann--{T}hompson width for hyperbolic
  3-manifolds.
\newblock {\em Pacific J. Math.}, 281(1):83--102, 2016.
\newblock \href {https://dx.doi.org/10.2140/pjm.2016.281.83}
  {\path{doi:10.2140/pjm.2016.281.83}}, \href
  {https://mathscinet.ams.org/mathscinet-getitem?mr=3459967}
  {\path{MR:3459967}}, \href {https://zbmath.org/?q=an:1335.57027}
  {\path{Zbl:1335.57027}}.

\bibitem{hoffoss2017morse}
D.~Hoffoss and J.~Maher.
\newblock Morse functions to graphs and topological complexity for hyperbolic
  3-manifolds, 2017.
\newblock 21 pages, 1 figure.
\newblock \href {https://arxiv.org/abs/1708.04140} {\path{arXiv:1708.04140}}.

\bibitem{hugh2007thinposition}
H.~Howards, Y.~Rieck, and J.~Schultens.
\newblock Thin position for knots and 3-manifolds: a unified approach.
\newblock In {\em Workshop on {H}eegaard {S}plittings}, volume~12 of {\em Geom.
  Topol. Monogr.}, pages 89--120. Geom. Topol. Publ., Coventry, 2007.
\newblock \href {https://dx.doi.org/10.2140/gtm.2007.12.89}
  {\path{doi:10.2140/gtm.2007.12.89}}, \href
  {https://mathscinet.ams.org/mathscinet-getitem?mr=2408244}
  {\path{MR:2408244}}, \href {https://zbmath.org/?q=an:1152.57012}
  {\path{Zbl:1152.57012}}.

\bibitem{huszar2019treewidth}
K.~Husz\'ar, J.~Spreer, and U.~Wagner.
\newblock On the treewidth of triangulated 3-manifolds.
\newblock {\em J. Comput. Geom.}, 10(2):70--98, 2019.
\newblock \href {https://dx.doi.org/10.20382/jogc.v10i2a5}
  {\path{doi:10.20382/jogc.v10i2a5}}, \href
  {https://mathscinet.ams.org/mathscinet-getitem?mr=4039886}
  {\path{MR:4039886}}, \href {https://zbmath.org/?q=an:07150581}
  {\path{Zbl:07150581}}.

\bibitem{ivanov2008computational}
S.~V. Ivanov.
\newblock The computational complexity of basic decision problems in
  3-dimensional topology.
\newblock {\em Geom. Dedicata}, 131:1--26, 2008.
\newblock \href {https://dx.doi.org/10.1007/s10711-007-9210-4}
  {\path{doi:10.1007/s10711-007-9210-4}}, \href
  {https://mathscinet.ams.org/mathscinet-getitem?mr=2369189}
  {\path{MR:2369189}}, \href {https://zbmath.org/?q=an:1146.57025}
  {\path{Zbl:1146.57025}}.

\bibitem{kinnersley1992vertex}
N.~G. Kinnersley.
\newblock The vertex separation number of a graph equals its path-width.
\newblock {\em Inf. Process. Lett.}, 42(6):345--350, 1992.
\newblock \href {https://dx.doi.org/10.1016/0020-0190(92)90234-M}
  {\path{doi:10.1016/0020-0190(92)90234-M}}, \href
  {https://mathscinet.ams.org/mathscinet-getitem?mr=1178214}
  {\path{MR:1178214}}, \href {https://zbmath.org/?q=an:0764.68121}
  {\path{Zbl:0764.68121}}.

\bibitem{kirby2004local}
R.~Kirby and P.~Melvin.
\newblock Local surgery formulas for quantum invariants and the {A}rf
  invariant.
\newblock In {\em Proc. {C}asson {F}est}, volume~7 of {\em Geom. Topol.
  Monogr.}, pages 213--233. Geom. Topol. Publ., Coventry, 2004.
\newblock \href {https://dx.doi.org/10.2140/gtm.2004.7.213}
  {\path{doi:10.2140/gtm.2004.7.213}}, \href
  {https://mathscinet.ams.org/mathscinet-getitem?mr=2172485}
  {\path{MR:2172485}}, \href {https://zbmath.org/?q=an:1087.57006}
  {\path{Zbl:1087.57006}}.

\bibitem{kleiner2008perelman}
B.~Kleiner and J.~Lott.
\newblock Notes on {P}erelman's papers.
\newblock {\em Geom. Topol.}, 12(5):2587--2855, 2008.
\newblock \href {https://dx.doi.org/10.2140/gt.2008.12.2587}
  {\path{doi:10.2140/gt.2008.12.2587}}, \href
  {https://mathscinet.ams.org/mathscinet-getitem?mr=2460872}
  {\path{MR:2460872}}, \href {https://zbmath.org/?q=an:1204.53033}
  {\path{Zbl:1204.53033}}.

\bibitem{kloks1994treewidth}
T.~Kloks.
\newblock {\em Treewidth: Computations and Approximations}, volume 842 of {\em
  Lect. Notes Comput. Sci.}
\newblock Springer, 1994.
\newblock \href {https://dx.doi.org/10.1007/BFb0045375}
  {\path{doi:10.1007/BFb0045375}}, \href
  {https://mathscinet.ams.org/mathscinet-getitem?mr=1312164}
  {\path{MR:1312164}}, \href {https://zbmath.org/?q=an:0825.68144}
  {\path{Zbl:0825.68144}}.

\bibitem{kuperberg2014knottedness}
G.~Kuperberg.
\newblock Knottedness is in {NP}, modulo {GRH}.
\newblock {\em Adv. Math.}, 256:493--506, 2014.
\newblock \href {https://dx.doi.org/10.1016/j.aim.2014.01.007}
  {\path{doi:10.1016/j.aim.2014.01.007}}, \href
  {https://mathscinet.ams.org/mathscinet-getitem?mr=3177300}
  {\path{MR:3177300}}, \href {https://zbmath.org/?q=an:1358.68138}
  {\path{Zbl:1358.68138}}.

\bibitem{kuperberg2019algorithmic}
G.~Kuperberg.
\newblock Algorithmic homeomorphism of 3-manifolds as a corollary of
  geometrization.
\newblock {\em Pacific J. Math.}, 301(1):189--241, 2019.
\newblock \href {https://dx.doi.org/10.2140/pjm.2019.301.189}
  {\path{doi:10.2140/pjm.2019.301.189}}, \href
  {https://mathscinet.ams.org/mathscinet-getitem?mr=4007377}
  {\path{MR:4007377}}.

\bibitem{lackenby2017conditionally}
M.~Lackenby.
\newblock Some conditionally hard problems on links and 3-manifolds.
\newblock {\em Discrete Comput. Geom.}, 58(3):580--595, 2017.
\newblock \href {https://dx.doi.org/10.1007/s00454-017-9905-8}
  {\path{doi:10.1007/s00454-017-9905-8}}, \href
  {https://mathscinet.ams.org/mathscinet-getitem?mr=3690662}
  {\path{MR:3690662}}, \href {https://zbmath.org/?q=an:1384.57010}
  {\path{Zbl:1384.57010}}.

\bibitem{lackenby2016efficient}
M.~Lackenby.
\newblock The efficient certification of knottedness and {Thurston} norm.
\newblock {\em Adv. Math.}, 387, 2021.
\newblock Article ID: 107796.
\newblock \href {https://dx.doi.org/10.1016/j.aim.2021.107796}
  {\path{doi:10.1016/j.aim.2021.107796}}, \href
  {https://zbmath.org/?q=an:07369650} {\path{Zbl:07369650}}.

\bibitem{maria2019polynomial}
C.~Maria and J.~Spreer.
\newblock A polynomial-time algorithm to compute {T}uraev--{V}iro invariants
  $\operatorname{TV}_{4,q}$ of 3-manifolds with bounded first {Betti} number.
\newblock {\em Found. Comput. Math.}, 20(5):1013--1034, 2020.
\newblock \href {https://dx.doi.org/10.1007/s10208-019-09438-8}
  {\path{doi:10.1007/s10208-019-09438-8}}, \href
  {https://zbmath.org/?q=an:1470.57035} {\path{Zbl:1470.57035}}.

\bibitem{matveev2007algorithmic}
S.~Matveev.
\newblock {\em Algorithmic Topology and Classification of $3$-Manifolds},
  volume~9 of {\em Algorithms Comput. Math.}
\newblock Springer, Berlin, 2nd edition, 2007.
\newblock \href {https://dx.doi.org/10.1007/978-3-540-45899-9}
  {\path{doi:10.1007/978-3-540-45899-9}}, \href
  {https://mathscinet.ams.org/mathscinet-getitem?mr=2341532}
  {\path{MR:2341532}}, \href {https://zbmath.org/?q=an:1128.57001}
  {\path{Zbl:1128.57001}}.

\bibitem{moise1952affine}
E.~E. Moise.
\newblock Affine structures in {$3$}-manifolds. {V}. {T}he triangulation
  theorem and {H}auptvermutung.
\newblock {\em Ann. Math. (2)}, 56:96--114, 1952.
\newblock \href {https://dx.doi.org/10.2307/1969769}
  {\path{doi:10.2307/1969769}}, \href
  {https://mathscinet.ams.org/mathscinet-getitem?mr=0048805}
  {\path{MR:0048805}}, \href {https://zbmath.org/?q=an:0048.17102}
  {\path{Zbl:0048.17102}}.

\bibitem{monien1988min}
B.~Monien and I.~H. Sudborough.
\newblock Min cut is {NP}-complete for edge weighted trees.
\newblock {\em Theoret. Comput. Sci.}, 58(1--3):209--229, 1988.
\newblock \href {https://dx.doi.org/10.1016/0304-3975(88)90028-X}
  {\path{doi:10.1016/0304-3975(88)90028-X}}, \href
  {https://mathscinet.ams.org/mathscinet-getitem?mr=963264} {\path{MR:963264}},
  \href {https://zbmath.org/?q=an:0657.68034} {\path{Zbl:0657.68034}}.

\bibitem{moriah1997heegaard}
Y.~Moriah and H.~Rubinstein.
\newblock Heegaard structures of negatively curved {$3$}-manifolds.
\newblock {\em Comm. Anal. Geom.}, 5(3):375--412, 1997.
\newblock \href {https://dx.doi.org/10.4310/CAG.1997.v5.n3.a1}
  {\path{doi:10.4310/CAG.1997.v5.n3.a1}}, \href
  {https://mathscinet.ams.org/mathscinet-getitem?mr=1487722}
  {\path{MR:1487722}}, \href {https://zbmath.org/?q=an:0890.57025}
  {\path{Zbl:0890.57025}}.

\bibitem{norine2006minorclosed}
S.~Norine, P.~D. Seymour, R.~Thomas, and P.~Wollan.
\newblock Proper minor-closed families are small.
\newblock {\em J. Comb. Theory, Ser. {B}}, 96(5):754--757, 2006.
\newblock \href {https://dx.doi.org/10.1016/j.jctb.2006.01.006}
  {\path{doi:10.1016/j.jctb.2006.01.006}}, \href
  {https://mathscinet.ams.org/mathscinet-getitem?mr=2236510}
  {\path{MR:2236510}}, \href {https://zbmath.org/?q=an:1093.05065}
  {\path{Zbl:1093.05065}}.

\bibitem{orlik2006seifert}
P.~Orlik.
\newblock {\em Seifert Manifolds}, volume 291 of {\em Lect. Notes Math.}
\newblock Springer-Verlag, Berlin-New York, 1972.
\newblock \href {https://dx.doi.org/10.1007/BFb0060329}
  {\path{doi:10.1007/BFb0060329}}, \href
  {https://mathscinet.ams.org/mathscinet-getitem?mr=0426001}
  {\path{MR:0426001}}, \href {https://zbmath.org/?q=an:0263.57001}
  {\path{Zbl:0263.57001}}.

\bibitem{ostrovskii2004minimal}
M.~I. Ostrovskii.
\newblock Minimal congestion trees.
\newblock {\em Discrete Math.}, 285(1--3):219--226, 2004.
\newblock \href {https://dx.doi.org/10.1016/j.disc.2004.02.009}
  {\path{doi:10.1016/j.disc.2004.02.009}}, \href
  {https://mathscinet.ams.org/mathscinet-getitem?mr=2062845}
  {\path{MR:2062845}}, \href {https://zbmath.org/?q=an:1051.05032}
  {\path{Zbl:1051.05032}}.

\bibitem{robertson1983graph}
N.~Robertson and P.~D. Seymour.
\newblock Graph minors. {I.} {E}xcluding a forest.
\newblock {\em J. Comb. Theory, Ser. {B}}, 35(1):39--61, 1983.
\newblock \href {https://dx.doi.org/10.1016/0095-8956(83)90079-5}
  {\path{doi:10.1016/0095-8956(83)90079-5}}, \href
  {https://mathscinet.ams.org/mathscinet-getitem?mr=723569} {\path{MR:723569}},
  \href {https://zbmath.org/?q=an:0521.05062} {\path{Zbl:0521.05062}}.

\bibitem{robertson1986graph}
N.~Robertson and P.~D. Seymour.
\newblock Graph minors. {II.} {A}lgorithmic aspects of tree-width.
\newblock {\em J. Algorithms}, 7(3):309--322, 1986.
\newblock \href {https://dx.doi.org/10.1016/0196-6774(86)90023-4}
  {\path{doi:10.1016/0196-6774(86)90023-4}}, \href
  {https://mathscinet.ams.org/mathscinet-getitem?mr=855559} {\path{MR:855559}},
  \href {https://zbmath.org/?q=an:0611.05017} {\path{Zbl:0611.05017}}.

\bibitem{rubinstein1995algorithm}
J.~H. Rubinstein.
\newblock An algorithm to recognize the {$3$}-sphere.
\newblock In {\em Proc. Int. Congr. Math., Z{\"u}rich, Switzerland, August
  3--11, 1994}, volume~1, pages 601--611. Birkh\"auser, 1995.
\newblock \href {https://dx.doi.org/10.1007/978-3-0348-9078-6_54}
  {\path{doi:10.1007/978-3-0348-9078-6_54}}, \href
  {https://mathscinet.ams.org/mathscinet-getitem?mr=1403961}
  {\path{MR:1403961}}, \href {https://zbmath.org/?q=an:0864.57009}
  {\path{Zbl:0864.57009}}.

\bibitem{scharlemann2002heegaard}
M.~Scharlemann.
\newblock Heegaard splittings of compact 3-manifolds.
\newblock In {\em Handbook of Geometric Topology}, pages 921--953.
  North-Holland, Amsterdam, 2001.
\newblock \href {https://dx.doi.org/10.1016/B978-044482432-5/50019-6}
  {\path{doi:10.1016/B978-044482432-5/50019-6}}, \href
  {https://mathscinet.ams.org/mathscinet-getitem?mr=1886684}
  {\path{MR:1886684}}, \href {https://zbmath.org/?q=an:0985.57005}
  {\path{Zbl:0985.57005}}.

\bibitem{scharlemann2016lecture}
M.~Scharlemann, J.~Schultens, and T.~Saito.
\newblock {\em Lecture Notes on Generalized {H}eegaard Splittings}.
\newblock World Scientific Publishing Co. Pte. Ltd., Hackensack, NJ, 2016.
\newblock \href {https://dx.doi.org/10.1142/10019} {\path{doi:10.1142/10019}},
  \href {https://mathscinet.ams.org/mathscinet-getitem?mr=3585907}
  {\path{MR:3585907}}, \href {https://zbmath.org/?q=an:1356.57004}
  {\path{Zbl:1356.57004}}.

\bibitem{scharlemann1992thin}
M.~Scharlemann and A.~Thompson.
\newblock Thin position for {$3$}-manifolds.
\newblock In {\em Geometric topology ({H}aifa, 1992)}, volume 164 of {\em
  Contemp. Math.}, pages 231--238. Amer. Math. Soc., Providence, RI, 1994.
\newblock \href {https://dx.doi.org/10.1090/conm/164/01596}
  {\path{doi:10.1090/conm/164/01596}}, \href
  {https://mathscinet.ams.org/mathscinet-getitem?mr=1282766}
  {\path{MR:1282766}}, \href {https://zbmath.org/?q=an:0818.57013}
  {\path{Zbl:0818.57013}}.

\bibitem{schleimer2011sphere}
S.~Schleimer.
\newblock Sphere recognition lies in {NP}.
\newblock In {\em Low-dimensional and symplectic topology}, volume~82 of {\em
  Proc. Sympos. Pure Math.}, pages 183--213. Amer. Math. Soc., Providence, RI,
  2011.
\newblock \href {https://dx.doi.org/10.1090/pspum/082/2768660}
  {\path{doi:10.1090/pspum/082/2768660}}, \href
  {https://mathscinet.ams.org/mathscinet-getitem?mr=2768660}
  {\path{MR:2768660}}, \href {https://zbmath.org/?q=an:1250.57024}
  {\path{Zbl:1250.57024}}.

\bibitem{schultens2014introduction}
J.~Schultens.
\newblock {\em Introduction to $3$-Manifolds}, volume 151 of {\em Grad. Stud.
  Math.}
\newblock Amer. Math. Soc., Providence, RI, 2014.
\newblock \href {https://dx.doi.org/10.1090/gsm/151}
  {\path{doi:10.1090/gsm/151}}, \href
  {https://mathscinet.ams.org/mathscinet-getitem?mr=3203728}
  {\path{MR:3203728}}, \href {https://zbmath.org/?q=an:1295.57001}
  {\path{Zbl:1295.57001}}.

\bibitem{scott2014homeomorphism}
P.~Scott and H.~Short.
\newblock The homeomorphism problem for closed 3-manifolds.
\newblock {\em Algebr. Geom. Topol.}, 14(4):2431--2444, 2014.
\newblock \href {https://dx.doi.org/10.2140/agt.2014.14.2431}
  {\path{doi:10.2140/agt.2014.14.2431}}, \href
  {https://mathscinet.ams.org/mathscinet-getitem?mr=3331689}
  {\path{MR:3331689}}, \href {https://zbmath.org/?q=an:1311.57025}
  {\path{Zbl:1311.57025}}.

\bibitem{seymour1994call}
P.~D. Seymour and R.~Thomas.
\newblock Call routing and the ratcatcher.
\newblock {\em Combinatorica}, 14(2):217--241, 1994.
\newblock \href {https://dx.doi.org/10.1007/BF01215352}
  {\path{doi:10.1007/BF01215352}}, \href
  {https://mathscinet.ams.org/mathscinet-getitem?mr=1289074}
  {\path{MR:1289074}}, \href {https://zbmath.org/?q=an:0799.05057}
  {\path{Zbl:0799.05057}}.

\bibitem{thilikos2000Carving}
D.~M. Thilikos, M.~J. Serna, and H.~L. Bodlaender.
\newblock Constructive linear time algorithms for small cutwidth and
  carving-width.
\newblock In {\em Proc. 11th Int. Conf. Algorithms Comput. ({ISAAC} 2000)},
  pages 192--203, 2000.
\newblock \href {https://dx.doi.org/10.1007/3-540-40996-3_17}
  {\path{doi:10.1007/3-540-40996-3_17}}, \href
  {https://mathscinet.ams.org/mathscinet-getitem?mr=1858364}
  {\path{MR:1858364}}, \href {https://zbmath.org/?q=an:1044.68709}
  {\path{Zbl:1044.68709}}.

\bibitem{thilikos2001polynomial}
D.~M. Thilikos, M.~J. Serna, and H.~L. Bodlaender.
\newblock A polynomial time algorithm for the cutwidth of bounded degree graphs
  with small treewidth.
\newblock In {\em Proc. 9th Annu. Eur. Symp. Algorithms ({ESA} 2001)}, pages
  380--390, 2001.
\newblock \href {https://dx.doi.org/10.1007/3-540-44676-1_32}
  {\path{doi:10.1007/3-540-44676-1_32}}, \href
  {https://mathscinet.ams.org/mathscinet-getitem?mr=1914050}
  {\path{MR:1914050}}, \href {https://zbmath.org/?q=an:1007.05091}
  {\path{Zbl:1007.05091}}.

\bibitem{thilikos2005cutwidth}
D.~M. Thilikos, M.~J. Serna, and H.~L. Bodlaender.
\newblock Cutwidth {I:} {A} linear time fixed parameter algorithm.
\newblock {\em J. Algorithms}, 56(1):1--24, 2005.
\newblock \href {https://dx.doi.org/10.1016/j.jalgor.2004.12.001}
  {\path{doi:10.1016/j.jalgor.2004.12.001}}, \href
  {https://mathscinet.ams.org/mathscinet-getitem?mr=2146375}
  {\path{MR:2146375}}, \href {https://zbmath.org/?q=an:1161.68856}
  {\path{Zbl:1161.68856}}.

\bibitem{thompson1994thin}
A.~Thompson.
\newblock Thin position and the recognition problem for {$S^3$}.
\newblock {\em Math. Res. Lett.}, 1(5):613--630, 1994.
\newblock \href {https://dx.doi.org/10.4310/MRL.1994.v1.n5.a9}
  {\path{doi:10.4310/MRL.1994.v1.n5.a9}}, \href
  {https://mathscinet.ams.org/mathscinet-getitem?mr=1295555}
  {\path{MR:1295555}}, \href {https://zbmath.org/?q=an:0849.57009}
  {\path{Zbl:0849.57009}}.

\bibitem{thurston2014three}
W.~P. Thurston.
\newblock {\em Three-Dimensional Geometry and Topology. {V}ol. 1}, volume~35 of
  {\em Princeton Math. Ser.}
\newblock Princeton Univ. Press, Princeton, NJ, 1997.
\newblock Edited by S. Levy.
\newblock \href {https://dx.doi.org/10.1515/9781400865321}
  {\path{doi:10.1515/9781400865321}}, \href
  {https://mathscinet.ams.org/mathscinet-getitem?mr=1435975}
  {\path{MR:1435975}}, \href {https://zbmath.org/?q=an:0873.57001}
  {\path{Zbl:0873.57001}}.

\bibitem{zentner2018integer}
R.~Zentner.
\newblock Integer homology {$3$}-spheres admit irreducible representations in
  {$\operatorname{SL}(2,\mathbb{C})$}.
\newblock {\em Duke Math. J.}, 167(9):1643--1712, 2018.
\newblock \href {https://dx.doi.org/10.1215/00127094-2018-0004}
  {\path{doi:10.1215/00127094-2018-0004}}, \href
  {https://mathscinet.ams.org/mathscinet-getitem?mr=3813594}
  {\path{MR:3813594}}, \href {https://zbmath.org/?q=an:06904637}
  {\path{Zbl:06904637}}.

\end{thebibliography}

\end{document}